\documentclass[12pt,a4paper]{article}
\usepackage{a4wide}
\usepackage{amsfonts,amsmath,mathtools,amssymb,mathrsfs}

\usepackage{latexsym}
\usepackage{times}

\usepackage{xcolor}
\usepackage[normalem]{ulem}

\author{K\'aroly J. B\"or\"oczky\footnote{{The author was supported by Hungarian National Research, Development and Innovation Office -- NKFIH grants 129630 and 132002.}}, Ferenc Fodor\thanks{The author was supported by Hungarian National Research, Development and Innovation Office -- NKFIH grants 129630.}, Daniel Hug}
\title{Strengthened inequalities for the mean width and the $\ell$-norm
}

\newcommand{\proofbox}{\mbox{ $\Box$}}
\newcommand{\R}{\mathbb{R}}
\newcommand{\N}{\mathbb{N}}

\DeclareMathOperator{\Id}{I}

\newtheorem{lemma}{Lemma}[section]
\newtheorem{theo}[lemma]{Theorem}

\newtheorem{coro}[lemma]{Corollary}

\newtheorem{prop}[lemma]{Proposition}

\allowdisplaybreaks

\begin{document}

\maketitle

\begin{abstract}
Barthe proved that the regular simplex maximizes the mean width of convex bodies whose John ellipsoid (maximal volume ellipsoid contained in the body)  is the Euclidean unit ball; or equivalently, the regular simplex maximizes the $\ell$-norm of convex bodies whose L\"owner  ellipsoid (minimal volume ellipsoid containing the body) is the Euclidean unit ball.
Schmuckenschl\"ager verified the reverse statement; namely, the regular simplex minimizes the mean width of convex bodies whose L\"owner ellipsoid is the Euclidean unit ball. 
Here we prove stronger stability versions of these results. We also consider related stability results for the mean width and the $\ell$-norm of the convex hull of the support of centered isotropic measures on the unit sphere. 
\end{abstract}


\section{Introduction}

In geometric inequalities and extremal problems, Euclidean balls and simplices often are the extremizers.
A classical example is the isoperimetric inequality which states that Euclidean balls have  smallest surface area among convex bodies (compact convex sets with non-empty interior) of given volume in Euclidean space $\R^n$, and Euclidean balls are the only minimizers. Another example is the Urysohn inequality which expresses the geometric fact that Euclidean balls minimize the mean width of convex bodies of given volume. To introduce the mean width,
let $\langle\cdot\,,\cdot\rangle$ and $\|\cdot\|$ denote the scalar product  and  Euclidean norm in $\R^n$, and
let $B^n$ be the Euclidean unit ball centred at the origin with $\kappa_n=V(B^n)=\pi^{n/2}/\Gamma(1+n/2)$, where
$V(\cdot)$ is the volume (Lebesgue measure) in $\R^n$. For a convex body $K$ in $\R^n$, the support function $h_K:\R^n\to\R$ of $K$ is defined by
$h_K(x)=\max_{y\in K}\langle x,y\rangle$ for $x\in\R^n$. Then the mean width of $K$ is given by
$$
W(K)=\frac1{n\kappa_n}\int_{S^{n-1}}(h_K(u)+h_K(-u))\,du,
$$
where the integration over the unit sphere $S^{n-1}$ is with respect to the $(n-1)$-dimensional Hausdorff measure ({that coincides with the spherical Lebesgue measure in this case}).

A prominent geometric extremal problem for which simplices are extremizers has been discovered and explored much more recently. First, recall that there exists a unique ellipsoid of maximal volume contained in $K$ (which is called the John ellipsoid of $K$), and a unique ellipsoid of minimal volume containing $K$ (which is called the L\"owner ellipsoid of $K$). It has been shown by Ball \cite{Bal91a} that simplices maximize  the volume of $K$
given the volume of the John ellipsoid of $K$, and thus simplices determine the extremal ``inner" volume ratio. For the dual problem, Barthe \cite{Bar98} proved that
simplices minimize  the volume of $K$
given the volume of the L\"owner ellipsoid of $K$, hence simplices determine the extremal ``outer" volume ratio (see also Lutwak, Yang, Zhang \cite{Lutwak0,Lutwak1}). In all these cases, equality was characterized by Barthe \cite{Bar98}.

In this paper, we consider the mean width and the so called $\ell$-norm.
To define the latter,  for a convex body $K\subset\R^n$ containing the origin in its interior, we set
$$
\|x\|_K=\min\{t\geq 0:\,x\in tK\}, \qquad  x\in\R^n.
$$
Furthermore, we write $\gamma_n$ for the standard Gaussian measure in $\R^n$ which has the density function
$x\mapsto \sqrt{2\pi}^{\,-n}e^{-{\|x\|^2}/{2}}$, $x\in\R^n$,  with respect to Lebesgue measure. Then the $\ell$-norm of $K$ is given by
$$
\ell(K)=\int_{\R^n}\|x\|_K\,\gamma_n(dx)=\mathbb{E} \| X\|_K,
$$
where $X$ is a Gaussian random vector with distribution $\gamma_n$.
If the polar body of $K$ is denoted by $K^\circ=\{x\in\R^n:\,\langle x,y\rangle\leq 1\;\forall y\in K\}$,
then we obtain the relation
\begin{equation}
\label{mean-ell}
\ell(K)=\mbox{$\frac{\ell(B^n)}2$}  W(K^\circ)
\end{equation}
with
$$
\lim_{n\to\infty}\frac{\ell(B^n)}{\sqrt{n}}=1.
$$
In addition, the $\ell$-norm of $K$ can be expressed in the form  (see Barthe \cite{Bar98b})
\begin{equation}
\label{ellGaussianint}
\ell(K)=\int_{\R^n}\mathbb{P}(\|X\|_K>t)\, dt
=\int_0^\infty(1-\gamma_n(tK))\,dt.
\end{equation}

 Let $\Delta_n$ be a regular simplex inscribed into $B^n$, and hence $\Delta_n^\circ$ is a regular simplex circumscribed around $B^n$. Theorem~\ref{meanw-non-sym} (i) is due to Barthe \cite{Bar98b}, and (ii) was proved by Schmuckenschl\"ager \cite{Sch99}. 

\begin{theo}[Barthe '98, Schmuckenschl\"ager '99]
\label{meanw-non-sym}
Let $K$ be a convex body in $\R^n$.
\begin{enumerate}
\item[{\rm (i)}] If $B^n\supset K$ is the L\"owner ellipsoid of $K$, then $\ell(K)\leq \ell(\Delta_n)$, and if 
 $B^n\subset K$ is the John ellipsoid of $K$, then $W(K)\leq W(\Delta_n^\circ)$. Equality holds in either case if and only if $K$ is a regular simplex.
\item[{\rm (ii)}] If $B^n\subset K$ is the John ellipsoid of $K$, then $\ell(K)\geq \ell(\Delta_n^\circ)$, and if 
$B^n\supset K$ is the L\"owner ellipsoid of $K$, then $W(K)\geq W(\Delta_n)$. Equality holds in either case if and only if $K$ is a regular simplex.
\end{enumerate}
\end{theo}

It follows from (\ref{mean-ell}) and the duality of L\"owner and John ellipsoids that the two statements in (i) are equivalent to each other, and the same is true for (ii). 

While a reverse form of the Urysohn inequality is still not known in general,  we recall that Giannopoulos, Milman, Rudelson \cite{GMR00} proved a reverse Urysohn inequality for zonoids, and Hug, Schneider \cite{HS11} established  reverse inequalities of other intrinsic and mixed volumes for zonoids and explored applications to stochastic geometry. A {related classical} open problem in convexity and probability theory is that
among all simplices contained in the Euclidean unit ball, the inscribed regular simplex
has the maximal mean width (see Litvak \cite{Lit18} for a comprehensive survey on this topic).

Let us discuss the range of $W(K)$ (and hence that of $\ell(K)$ by \eqref{mean-ell}) in Theorem~\ref{meanw-non-sym}.
If $K$ is a convex body in $\R^n$ whose L\"owner ellipsoid is $B^n$,
then the monotonicity of the mean width and Theorem~\ref{meanw-non-sym} (i) yield
$$
W(\Delta_n)\leq W(K)\leq W(B^n)=2,
$$
where, according to B\"or\"oczky \cite{Bor94}, we have
$$
W(\Delta_n)\sim 4\sqrt{\frac{2\ln n}n}\mbox{ \ as $n\to\infty$}.
$$
In addition, if $K$ is a convex body in $\R^n$ whose John ellipsoid is $B^n$,
then
$$2=W(B^n)\leq W(K)\leq W(\Delta_n^\circ)$$
with $W(\Delta_n^\circ)\sim 4\sqrt{2n\ln n}$.

An important concept in the proof of Theorem~\ref{meanw-non-sym} is the notion of an isotropic measure on the unit sphere. Following Giannopoulos, Papadimitrakis \cite{GianPapa1999} and
Lutwak, Yang, Zhang \cite{Lutwak1}, we call
a Borel measure $\mu$ on the unit sphere $S^{n-1}$ isotropic if
\begin{equation}
\label{isotropic-def}
\int_{S^{n-1}}u\otimes u\, \mu(du)=\Id_n,
\end{equation}
where $\Id_n$ is the {identity map (or the identity matrix)}. Condition \eqref{isotropic-def} is equivalent to
$$
\langle x,x\rangle=\int_{S^{n-1}}\langle u,x\rangle^2\, \mu(du) \mbox{ \ for  $x\in\R^n$}.
$$
In this case, equating traces of the two sides of (\ref{isotropic-def}), we obtain that 
$\mu(S^{n-1})=n$. 
In addition, we say that the isotropic measure $\mu$ on $S^{n-1}$ is centered if
$$
\int_{S^{n-1}}u\, \mu(du)=o.
$$
We observe that if $\mu$ is a centered isotropic measure on $S^{n-1}$, then
$|{\rm supp}\,\mu|\geq n+1$, with equality if and only if $\mu$ is concentrated on the vertices of some regular simplex and each vertex has  measure $n/(n+1)$.

We recall that isotropic measures on $\R^n$ play a central role in the
KLS conjecture by Kannan, Lov\'asz and   Simonovits
\cite{KLM95} as well as in the analysis of Bourgain's hyperplane conjecture (slicing problem); see, for instance,
Barthe and  Cordero-Erausquin \cite{BCE13}, Guedon and  Milman \cite{GuM11},
Klartag \cite{Kla09}, Artstein-Avidan,  Giannopoulos, Milman \cite{AGM15}
and Alonso-Guti\'errez, Bastero \cite{AGB15}.

The emergence of isotropic measures on $S^{n-1}$ arises from Ball's crucial insight that John's characteristic condition \cite{Joh37,Joh48} for a convex body to have the unit ball as its John or L\"owner ellipsoid
 (see \cite{Bal89,Bal91a}) can be used to give the Brascamp-Lieb inequality a convenient form which is ideally suited  for many
 geometric applications
(see Section~\ref{secbrascamp-lieb}). John's characteristic condition (with the proof of the equivalence completed by Ball \cite{Bal92}) states  that $B^n$ is the John ellipsoid of a convex body $K$ containing  $B^n$ if and only if there exist distinct unit vectors $u_1,\ldots,u_k\in \partial K\cap S^{n-1}$ and $c_1,\ldots,c_k>0$ such that
\begin{align}
\label{John-iso}
\sum_{i=1}^kc_iu_i\otimes u_i&=\Id_n,\\
\label{John-centered}
\sum_{i=1}^kc_iu_i&=o.
\end{align}
In particular, the measure $\mu$ on $S^{n-1}$ with support $\{u_1,\ldots,u_k\}$ and $\mu(\{u_i\})=c_i$ for
$i=1,\ldots,k$ is isotropic and centered. In addition, $B^n$ is the L\"owner ellipsoid of a convex body
$K\subset B^n$ if and only if there exist $u_1,\ldots,u_k\in \partial K\cap S^{n-1}$ and $c_1,\ldots,c_k>0$ satisfying
(\ref{John-iso}) and
(\ref{John-centered}).
According to John \cite{Joh48} (see also Gruber, Schuster \cite{GrS05}), we may assume
that
$k\leq n(n+3)/2$ in (\ref{John-iso}) and
(\ref{John-centered}).
It follows from John's characterization that $B^n$ is the L\"owner ellipsoid of a convex body
$K\subset B^n$ if and only if
$B^n$ is the John ellipsoid of $K^\circ$.

The finite Borel measures on $S^{n-1}$ which have an isotropic linear image are characterized by
B\"or\"oczky, Lutwak, Yang and  Zhang \cite{BLYZ15}, building on earlier work by
Carlen,
and Cordero-Erausquin \cite{CCE09},  Bennett, Carbery, Christ and Tao \cite{BCCT08} and Klartag \cite{Kla10}.

We write ${\rm conv}\,X$ to denote the convex hull of a set $X\subset \R^n$. We observe that if $\mu$ is a centered isotropic measure on $S^{n-1}$, then $o\in{\rm int}\,Z_\infty(\mu)$ for
$$
Z_\infty(\mu)={\rm conv}\,{\rm supp}\,\mu.
$$
For the present purpose, the study of $Z_\infty(\mu)$ can be reduced to discrete measures,
as
Lemma~10.1 in B\"or\"oczky, Hug \cite{BoH17} states that for any centered isotropic measure $\mu$, there exists a discrete centered isotropic measure $\mu_0$ on $S^{n-1}$ whose support is contained in the support of $\mu$ (see Lemma~\ref{discrete-iso}).
It follows  that
Theorem~\ref{meanw-non-sym} is equivalent to the following statements about isotropic measures proved by Li, Leng \cite{LiL12}.

\begin{theo}[Li, Leng '12]
\label{meanw-isotropic}
If $\mu$ is a centered isotropic measure on $S^{n-1}$, then $\ell(Z_\infty(\mu))\leq \ell(\Delta_n)$, $W(Z_\infty(\mu)^\circ)\leq W(\Delta_n^\circ)$, $\ell(Z_\infty(\mu)^\circ)\geq \ell(\Delta_n^\circ)$ and 
 $W(Z_\infty(\mu))\geq W(\Delta_n)$, 
with equality in either case if and  only if $|{\rm supp}\,\mu|=n+1$.
\end{theo}

Results similar to Theorem~\ref{meanw-isotropic} are proved by Ma \cite{TMa17} in the $L_p$ setting.

The main goal of the present paper is to provide stronger stability versions of Theorem~\ref{meanw-non-sym} and 
Theorem \ref{meanw-isotropic}. Since {our} results {use the notion of distance} between  convex bodies (and to fix the notation), we recall that the  distance between compact subsets $X$ and $Y$ of $\R^n$ is  measured in terms of the Hausdorff distance {defined} by
$$
\delta_H(X,Y)=\max\{\max_{y\in Y}d(y,X),\max_{x\in X}d(x,Y)\}, 
$$
where $d(x,Y)=\min\{\|x-y\|:y\in Y\}$. The Hausdorf distance defines a metric on the set of non-empty compact subsets of $\R^n$.

In addition, for convex bodies $K$ and $C$, the symmetric difference distance of $K$ and $C$ is
the volume of their symmetric difference; namely,
$$
\delta_{\rm vol}(K,C)=V(K\setminus C)+V(C\setminus K).
$$
Clearly, the symmetric difference distance also defines a metric on the set of convex bodies in $\R^n$. Both  metrics 
induce the same topology on the space of convex bodies, but not the same uniform structure. 

Let $O(n)$ denote the orthogonal group (rotation group) of $\R^n$.

\begin{theo}
\label{Lowner-stab}
Let $B^n$ be the {L\"owner} ellipsoid of a convex body $K\subset B^n$ in $\R^n$,
let $c=n^{26n}$ and let $\varepsilon\in(0,1)$.
If $\ell(K)\geq (1-\varepsilon)\ell(\Delta_n)$, then
there exists a $T\in O(n)$ such that
\begin{enumerate}
\item[{\rm (i)}] $\delta_{\rm vol}(K,T\Delta_n)\leq c\, \sqrt[4]{\varepsilon},$
\item[{\rm (ii)}] $\delta_H(K,T\Delta_n)\leq c\, \sqrt[4]{\varepsilon}$.
\end{enumerate}
\end{theo}

\begin{theo}
\label{John-stab}
Let $B^n$ be the John ellipsoid of a convex body $K\supset B^n$ in $\R^n$
and let $\varepsilon>0$.
 If  $\ell(K)\leq (1+\varepsilon)\ell(\Delta_n^\circ)$, then
there exists a $T\in O(n)$ such that
\begin{enumerate}
\item[{\rm (i)}] $\delta_{\rm vol}(K,T\Delta_n^\circ)\leq c \,\sqrt[4]{\varepsilon}$ \,for $c=n^{27n}$,
\item[{\rm (ii)}] $\delta_H(K,T\Delta_n^\circ)\leq c\,\sqrt[4n]{\varepsilon}$\, for $c=n^{27}$.
\end{enumerate}
\end{theo}

Let us consider the optimality of the order of the estimates in Theorems~\ref{Lowner-stab} and \ref{John-stab}.
For  Theorem~\ref{Lowner-stab} (i) and (ii),
adding an {$(n+2)$nd} vertex $v_{n+2}\in S^{n-1}$ to the $n+1$ vertices $v_1,\ldots,v_{n+1}$ of $\Delta_n$
with $\angle(v_{n+2},v_1)=c_1\varepsilon$ for suitable $c_1>0$ depending on $n$
and $v_1$ lying on the geodesic arc on $S^{n-1}$ connecting $v_2$ and
$v_{n+2}$, the polytope $K={\rm conv}\{v_1,\ldots,v_{n+2}\}$ satisfies
$\ell(K)\geq (1-\varepsilon)\ell(\Delta_n)$ on the one hand, and
 $\delta_{\rm vol}(K,T\Delta_n)\geq c_2 \varepsilon$ and $\delta_H(K,T\Delta_n)\geq c_2 \varepsilon$
for suitable $c_2>0$ depending on $n$ and for any $T\in O(n)$ on the other hand.
Similarly, using the polar of this polytope $K$ for Theorem~\ref{John-stab} (i), possibly after decreasing $c_1$, we have
$\ell(K^\circ)\leq (1+\varepsilon)\ell(\Delta_n^\circ)$ while
$\delta_{\rm vol}(K^\circ,T\Delta_n^\circ)\geq c_3\varepsilon$
for suitable $c_3>0$ depending on $n$ and for any $T\in O(n)$.
Finally, we consider the optimality of Theorem~\ref{John-stab} (ii). Cutting off $n+1$ regular simplices of edge length
$c_4\sqrt[n]{\varepsilon}$ at the vertices of
$\Delta_n^\circ$, for suitable $c_4>0$ depending on $n$,  results in a polytope $\widetilde{K}$
satisfying $\ell(\widetilde{K})\leq (1+\varepsilon)\ell(\Delta_n^\circ)$ and
$\delta_H(\widetilde{K},T\Delta_n^\circ)\geq c_5\sqrt[n]{\varepsilon}$
for suitable $c_5>0$ depending on $n$ and for any $T\in O(n)$.

We did not make an attempt to optimize the {constants $c$ that depend on $n$}, but observe that
the {$c$} is polynomial in $n$ in Theorem~\ref{John-stab} (ii).

In the case of the mean width, we have the following stability versions of Theorem~\ref{meanw-non-sym}.

\begin{coro}
\label{mean-width-stab}
Let $K$ be convex body  in $\R^n$.
\begin{enumerate}
\item[{\rm (i)}] If $B^n$ is the John ellipsoid of $K\supset B^n$ and $W(K)\geq (1-\varepsilon)W(\Delta_n^\circ)$
for some $\varepsilon\in(0,1)$, then
there exists a $T\in O(n)$ such that
$\delta_H(K,T\Delta_n^\circ)\leq c\sqrt[4]{\varepsilon}$ for $c=n^{27n}$.
\item[{\rm (ii)}] If $B^n$ is the {L\"owner} ellipsoid of $K\subset B^n$ and $W(K)\leq (1+\varepsilon)W(\Delta_n)$
for some $\varepsilon>0$, then there exists a $T\in O(n)$ such that
$\delta_H(K,T\Delta_n)\leq c\sqrt[4n]{\varepsilon}$ for $c=n^{29}$.

\end{enumerate}
\end{coro}

For the optimality of Corollary~\ref{mean-width-stab} (i), cutting off $n+1$ regular simplices of edge length
$c_1\varepsilon$ at the vertices of
$\Delta_n^\circ$ for suitable $c_1>0$ depending on $n$  results in a polytope $K$
satisfying $W(K)\geq (1-\varepsilon)W(\Delta_n^\circ)$ and
$\delta_H(K,T\Delta_n^\circ)\geq c_2\varepsilon$
for suitable $c_2>0$ depending on $n$ and for any $T\in O(n)$. Concerning Corollary~\ref{mean-width-stab} (ii),
let $v_1,\ldots,v_{n+1}$ be the vertices of $\Delta_n$, and let $\widetilde{K}$ be the polytope whose vertices are
$v_i,-(\frac1n+c_3\sqrt[n]{\varepsilon})v_i$ for $i=1,\ldots,n+1$ for suitable $c_3>0$ depending on $n$
in a way such that $W(\widetilde{K})\leq (1+\varepsilon)W(\Delta)$. It follows that
$\delta_H(K,T\Delta_n)\geq c_4\sqrt[n]{\varepsilon}$ for any $T\in O(n)$ and
for a suitable $c_4>0$ depending on $n$.

\bigskip

We also have the following stronger form of Theorem ~\ref{meanw-isotropic} in the form of stability {statements}.

\begin{theo}
\label{meanw-isotropicstab}
Let $\mu$ be a centered isotropic measure on the unit sphere $S^{n-1}$, let $c=n^{28n}$, and let $\varepsilon\in(0,1)$.
If one of the conditions
\begin{enumerate}
\item[{\rm (a)}] $\ell(Z_\infty(\mu))\geq (1-\varepsilon)\ell(\Delta_n)$ or
\item[{\rm (b)}] $W(Z_\infty(\mu)^\circ)\geq (1-\varepsilon)W(\Delta_n^\circ)$ or
\item[{\rm (c)}] $\ell(Z_\infty(\mu)^\circ)\leq (1+\varepsilon)\ell(\Delta_n^\circ)$ or
\item[{\rm (d)}] $W(Z_\infty(\mu))\leq (1+\varepsilon)W(\Delta_n)$
\end{enumerate}
is satisfied, then
there exists a regular simplex with vertices $w_1,\ldots,w_{n+1}\in S^{n-1}$ such that
$$
\delta_H({\rm supp}\,\mu,\{w_1,\ldots,w_{n+1}\})\leq c\,\varepsilon^{\frac14}.
$$
\end{theo}

The proofs of Theorem~\ref{Lowner-stab} and Theorem \ref{meanw-isotropicstab} (a), (b) are based on Proposition~\ref{uiwjdist},
which is the special case of Theorem \ref{meanw-isotropicstab} (a) for a discrete measure. In addition, a new stability version of Barthe's reverse of the Brascamp-Lieb inequality is required for a special parametric class of functions, which is derived in Section~\ref{sec6}. In a similar vein, the proofs of Theorem~\ref{John-stab} and Theorem \ref{meanw-isotropicstab} (c), (d) are based on Proposition~\ref{uiwjdist0}, which is the special case of Theorem \ref{meanw-isotropicstab} (c) for a discrete measure. In addition, we {use} and derive a stability version of the Brascamp-Lieb inequality  for a special parametric class of functions (see also Section~\ref{sec6}).

We note that our arguments are based on the rank one {geometric} Brascamp-Lieb and reverse Brascamp-Lieb inequalities
(see Section~\ref{secbrascamp-lieb}), and their stability versions in a special case (see Section~\ref{sec6}). Unfortunately, no quantitative stability version of the Brascamp-Lieb and reverse Brascamp-Lieb inequalities are known in general
(see Bennett, Bez, Flock, Lee \cite{BBFL18} for a certain weak stability version for higher ranks).
 On the other hand, in the case
of the Borell-Brascamp-Lieb inequaliy (see Borell \cite{Bor75},  Brascamp, Lieb \cite{BrL76} and Balogh, {Krist\'aly} \cite{BaK18}),
 stability versions were {proved} by Ghilli, Salani \cite{GhS17} and Rossi, Salani \cite{RoS171}.

\section{Discrete isotropic measures and the  (reverse) Brascamp-Lieb inequality
	}
\label{secbrascamp-lieb}

For the purposes of this paper, the study of $Z_\infty(\mu)$ for centered isotropic
measures on $S^{n-1}$ can be reduced to the case when $\mu$ is discrete. Writing $|X|$ for the cardinality of a finite set $X$,
we recall that
Lemma~10.1 in B\"or\"oczky, Hug \cite{BoH17} states that for any centered isotropic measure $\mu$, there exists a discrete centered isotropic measure $\mu_0$ on $S^{n-1}$ with ${\rm supp}\,\mu_0\subset {\rm supp}\,\mu$
and $|{\rm supp}\,\mu_0|\leq \frac{n(n+3)}2+1$. We use this statement in the following form.

\begin{lemma}
\label{discrete-iso}
For any centered isotropic measure $\mu$ on $S^{n-1}$,
 there exists a  discrete centered isotropic measure
$\mu_0$ on $S^{n-1}$ such that
$$
{\rm supp}\,\mu_0\subset {\rm supp}\,\mu\mbox{ \ and \ }
|{\rm supp}\,\mu_0|\leq 2n^2.
$$
\end{lemma}

The rank one geometric Brascamp-Lieb inequality \eqref{BL} was identified by   Ball \cite{Bal89}
as an important case
of the rank one Brascamp-Lieb inequality proved originally
by Brascamp, Lieb \cite{BrL76}. In addition,
 the reverse Brascamp-Lieb inequality \eqref{RBL} is due to Barthe \cite{Bar97,Bar98}.
To set up \eqref{BL} and \eqref{RBL},
let the distinct unit vectors $u_1,\ldots,u_k\in S^{n-1}$  and $c_1,\ldots,c_k>0$ satisfy 
\begin{equation}
\label{ciui}
\sum_{i=1}^kc_i u_i\otimes u_i=\Id_n.
\end{equation}
If $f_1,\ldots,f_k$ are non-negative measurable functions on $\R$, then
the Brascamp-Lieb inequality  states that
\begin{equation}
\label{BL}
\int_{\R^n}\prod_{i=1}^kf_i(\langle x,u_i\rangle)^{c_i}\,dx \leq
\prod_{i=1}^k\left(\int_{\R}f_i\right)^{c_i}, 
\end{equation}
and the reverse Brascamp-Lieb inequality is given by
\begin{equation}
\label{RBL}
\int_{\R^n}^*\sup_{x=\sum_{i=1}^kc_i\theta_iu_i}\prod_{i=1}^kf_i(\theta_i)^{c_i}\,dx \geq
\prod_{i=1}^k\left(\int_{\R}f_i\right)^{c_i},
\end{equation}
where the star on the left-hand-side denotes the upper integral.
Here we always assume that $\theta_1,\ldots,\theta_k\in\R$ in (\ref{RBL}). We note that $\theta_1,\ldots,\theta_k$ are unique if $k=n$ and hence $u_1,\ldots,u_n$ is an orthonormal basis.

\bigskip

It was proved by Barthe \cite{Bar98} that equality in (\ref{BL}) or in (\ref{RBL}) implies that if none of the functions $f_i$ is identically zero or a scaled version of a  Gaussian, then there exists an origin symmetric regular crosspolytope in $\R^n$ such that $u_1,\ldots,u_k$ lie among its vertices. Conversely, we note that equality holds in (\ref{BL}) and (\ref{RBL}) if either each $f_i$ is a scaled version of the same centered Gaussian,
or if $k=n$ and $u_1,\ldots,u_n$ form an  orthonormal basis.

For a detailed discussion of the rank one Brascamp-Lieb inequality, we refer to
Carlen, Cordero-Erausquin \cite{CCE09}. The higher rank case,
due to Lieb \cite{Lie90}, is reproved and further explored by Barthe \cite{Bar98}.
Equality in the general version of the Brascamp-Lieb inequality is clarified by
Bennett, Carbery, Christ, Tao \cite{BCCT08}. In addition, Barthe, Cordero-Erausquin, Ledoux,
Maurey (see \cite{BCLM11}) develop an approach for the Brascamp-Lieb inequality via Markov semigroups in a quite general framework.

The fundamental papers by Barthe \cite{Bar97, Bar98}  provided concise proofs of \eqref{BL} and \eqref{RBL}
based on mass transportation (see Ball \cite{Bal03} for a sketch in the case of \eqref{BL}).
Actually, the reverse Brascamp-Lieb inequality \eqref{RBL} seems to be the first inequality whose original proof is via mass transportation.
During the argument in Barthe \cite{Bar98}, the following four observations due to K.M.~Ball \cite{Bal89}
(see also \cite{Bar98} for a simpler proof of (i)) play crucial roles:
If $k\geq n$, $c_1,\ldots,c_k>0$ and  $u_1,\ldots,u_k\in S^{n-1}$ satisfy
(\ref{ciui}), then
\begin{description}
\item{(i)} for any $t_1,\ldots,t_k>0$, we have
\begin{equation}
\label{BallBarthe}
\det \left(\sum_{i=1}^kt_ic_i u_i\otimes u_i\right)\geq \prod_{i=1}^k t_i^{c_i},
\end{equation}

\item{(ii)}
for $z=\sum_{i=1}^kc_i\theta_i u_i$ with
$\theta_1,\ldots,\theta_k\in\R$, we have
\begin{equation}
\label{BLRBLquad}
\|z\|^2\le \sum_{i=1}^kc_i\theta_i^2,
\end{equation}

\item{(iii)} for $i=1,\ldots,k$, we have
\begin{equation*}
c_i\leq 1,
\end{equation*}

\item{(iv)} {and it holds that}
\begin{equation}
\label{cisum}
c_1+\cdots+c_k=n.
\end{equation}

\end{description}
Inequality (\ref{BallBarthe}) is called the Ball-Barthe inequality by Lutwak, Yang, Zhang \cite{Lutwak1} and Li, Leng \cite{LiL12}.

\section{Review of the proof of the  (reverse) Brascamp-Lieb inequality 
	if all $f_i=f$ and $f$ is log-concave}
\label{secBLRBL}

Let $g(t)=\sqrt{{2\pi}}^{-1}\,e^{-t^2/2}$, $t\in\R$, be the standard Gaussian density (mean zero, variance one), and let $f$
be a probability density function on $\R$ (here we restrict to log-concave {functions} to avoid  differentiability issues). Let $T$ and $S$ be the transportation maps which are determined by
$$
\int_{-\infty}^x f=\int_{-\infty}^{T(x)} g\quad
\mbox{ \ and \ }\quad \int_{-\infty}^{S(y)} f=\int_{-\infty}^{y} g.
$$
{Henceforth,} we do not write the arguments and the Lebesgue measure {in} the integral if the meaning
of the integral is unambiguous.
As $f$ is log-concave, there exists an open interval $I$ such that $f$ is positive on $I$ and zero on the complement of the closure of $I$, and $T: I\to\R$ and $S:\, \R\to I$ are inverses of each other. In addition, for $x\in I$ and $y\in\R$ we have
\begin{equation}
\label{masstrans}
f(x)=g(T(x))\, T'(x) \quad \mbox{ \ and \ } \quad g(y)=f(S(y))\, S'(y).
\end{equation}

For
$$
{\cal C}=\{x\in\R^n:\, \langle u_i,x\rangle\in I\;\text{ for } i=1,\ldots,k\},
$$
we consider the transformation $\Theta:{\cal C}\to\R^n$ with
$$
\Theta(x)=\sum_{i=1}^kc_i\, T(\langle u_i,x\rangle )\,u_i,\qquad x\in {\cal C},
$$
which satisfies
$$
d\Theta(x)=\sum_{i=1}^kc_i\, T'(\langle u_i,x\rangle )\,u_i\otimes u_i.
$$
It is {known} that $d\Theta$ is positive definite and $\Theta:{\cal C}\to\R^n$ is
 injective (see \cite{Bar97,Bar98}).
Therefore, using first \eqref{masstrans},  then (i) with $t_i=T'(\langle u_i,x\rangle)$,
and then the definition of $\Theta$ and  (ii), the following argument leads to the Brascamp-Lieb inequality
in this special case:
\begin{align*}
\nonumber
&\int_{\R^n}\prod_{i=1}^kf(\langle u_i,x\rangle)^{c_i}\,dx\\
&\qquad=
\int_{{\cal C}}\prod_{i=1}^kf(\langle u_i,x\rangle)^{c_i}\,dx\\
&\qquad=
\int_{{\cal C}}\left(\prod_{i=1}^kg(T(\langle u_i,x\rangle))^{c_i}\right)
\left(\prod_{i=1}^kT'(\langle u_i,x\rangle)^{c_i}\right)\,dx\\
&\qquad\leq \frac{1}{(2\pi)^{\frac{n}2}}\int_{{\cal C}}\left(\prod_{i=1}^ke^{-c_iT(\langle u_i,x\rangle)^2/2}\right)
 \det\left(\sum_{i=1}^kc_iT'(\langle u_i,x\rangle )\,u_i\otimes u_i\right)\,dx\\
&\qquad\leq  \frac{1}{(2\pi)^{\frac{n}2}}\int_{{\cal C}}e^{-\|\Theta(x)\|^2/2}\det\left( d\Theta(x)\right)\,dx\\
\nonumber
&\qquad\leq \frac{1}{(2\pi)^{\frac{n}2}} \int_{\R^n}e^{-\|y\|^2/2}\,dy=1.
\end{align*}
We note that the Brascamp-Lieb inequality (\ref{BLf}) for an arbitrary non-negative log-concave function $f$ follows by scaling; namely, (iv) implies
\begin{equation}
\label{BLf}
\int_{\R^n}\prod_{i=1}^kf(\langle x,u_i\rangle)^{c_i}\,dx\leq
\left(\int_{\R}f\right)^{n}.
\end{equation}

For the reverse Brascamp-Lieb inequality, we observe that
$$
d\Psi(x)=\sum_{i=1}^kc_iS'(\langle u_i,x\rangle )\,u_i\otimes u_i
$$
holds for the differentiable map $\Psi:\R^n\to \R^n$ with
$$
\Psi(x)=\sum_{i=1}^kc_iS(\langle u_i,x\rangle )\,u_i,\qquad x\in\R^n.
$$
In particular, $d\Psi$ is positive definite and $\Psi:\R^n\to\R^n$ is
 injective (see \cite{Bar97,Bar98}).
Therefore (i) and (\ref{masstrans}) lead to
\begin{align*}
\nonumber
&\int_{\R^n}^*\sup_{x=\sum_{i=1}^kc_i\theta_iu_i}\prod_{i=1}^kf(\theta_i)^{c_i}\,dx\\
&\qquad \geq
\int_{\R^n}^*\left(\sup_{\Psi(y)=\sum_{i=1}^kc_i\theta_iu_i}\prod_{i=1}^kf(\theta_i)^{c_i}\right)
\det\left( d\Psi(y)\right)\,dy\\
&\qquad\geq  \int_{\R^n}\left(\prod_{i=1}^kf(S(\langle u_i,y\rangle))^{c_i} \right)
\det\left(\sum_{i=1}^kc_iS'(\langle u_i,y\rangle )\,u_i\otimes u_i\right)\,dy\\
&\qquad\geq \int_{\R^n}\left(\prod_{i=1}^kf(S(\langle u_i,y\rangle))^{c_i}\right)
\left(\prod_{i=1}^kS'(\langle u_i,y\rangle)^{c_i}\right)\,dy\\
\nonumber
&\qquad= \int_{\R^n}\left(\prod_{i=1}^kg(\langle u_i,y\rangle)^{c_i}\right)
\,dy=  \frac{1}{(2\pi)^{\frac{n}2}}\int_{\R^n}e^{- \|y\|^2/2}\,dy=1.
\end{align*}
Again, the reverse Brascamp-Lieb inequality (\ref{RBLf}) for an arbitrary non-negative log-concave function $f$ follows by scaling and (iv); namely,
\begin{equation}
\label{RBLf}
{\int^*_{\R^n}}\sup_{x=\sum_{i=1}^kc_i\theta_iu_i}\prod_{i=1}^kf(\theta_i)^{c_i}\,dx\geq
\left(\int_{\R}f\right)^{n}.
\end{equation}

We observe that (i) shows that the optimal factor in the geometric Brascamp-Lieb inequaliy and
in its reverse form is $1$.

\section{Observations on the stability of the Brascamp-Lieb inequality and its reverse}

This section summerizes certain stability forms of the Ball-Barthe inequality (\ref{BallBarthe})
based on work in B\"or\"oczky, Hug \cite{BoH17}.
The first step is a stability version of the Ball-Barthe inequality (\ref{BallBarthe}) proved in \cite{BoH17}.

\begin{lemma}
\label{Ball-Barthe-stab0}
If $k\geq n+1$, $t_1,\ldots,t_k>0$, $c_1,\ldots,c_k>0$ and  $u_1,\ldots,u_k\in S^{n-1}$ satisfy
(\ref{ciui}), then
$$
\det\left( \sum_{i=1}^kt_ic_iu_i\otimes u_i\right)\geq \theta^* \prod_{i=1}^k t_i^{c_i},
$$
where
\begin{align*}
\theta^*&=  1+\frac12\sum_{1\leq i_1<\ldots<i_n\leq k}
c_{i_1}\cdots c_{i_n}\det[u_{i_1},\ldots,u_{i_n}]^2
\left(\frac{\sqrt{t_{i_1}\cdots  t_{i_n}}}{t_0}-1\right)^2,\\
t_0&=\sqrt{\sum_{1\leq i_1<\ldots<i_n\leq k}
t_{i_1}\cdots t_{i_n} c_{i_1}\cdots c_{i_n}\det[u_{i_1},\ldots,u_{i_n}]^2}.
\end{align*}
\end{lemma}

In order to estimate $\theta^*$ from below, we use the following observation from \cite{BoH17}.

\begin{lemma}
\label{xab}
If $a,b,x>0$, then
$$
(xa-1)^2+(xb-1)^2\geq \frac{(a^2-b^2)^2}{2(a^2+b^2)^2}.
$$
\end{lemma}

The combination of  Lemma~\ref{Ball-Barthe-stab0} and Lemma~\ref{xab} implies the following stability version of the Ball-Barthe inequality (\ref{BallBarthe}) {that is easier to use}.

\begin{coro}
\label{Ball-Barthe-stab}
If $k\geq n+1$, $t_1,\ldots,t_k>0$, $c_1,\ldots,c_k>0$ and  $u_1,\ldots,u_k\in S^{n-1}$ satisfy
(\ref{ciui}), and there exist $\beta>0$ and $n+1$ indices $ \{i_1,\ldots,i_{n+1}\}\subset \{1,\ldots,k\}$ such that
\begin{align*}
c_{i_1}\cdots c_{i_n}\det[u_{i_1},\ldots,u_{i_n}]^2&\geq \beta,\\
c_{i_2}\cdots c_{i_{n+1}}\det[u_{i_2},\ldots,u_{i_{n+1}}]^2&\geq \beta,
\end{align*}
then
$$
\det\left( \sum_{i=1}^kt_ic_iu_i\otimes u_i\right)\geq \left( 1+\frac{\beta(t_{i_1}-t_{i_{n+1}})^2}{4(t_{i_1}+t_{i_{n+1}})^2}\right) \prod_{i=1}^k t_i^{c_i}.
$$
\end{coro}

We may assume that $k\leq 2n^2$ (see Lemma~\ref{discrete-iso}), and thus  the following  observation from \cite{BoH17} can be used to estimate $\beta$
in Corollary~\ref{Ball-Barthe-stab} from below.

\begin{lemma}
\label{ciuibig}
If $k\geq n$, $c_1,\ldots,c_k>0$ and  $u_1,\ldots,u_k\in S^{n-1}$ satisfy
(\ref{ciui}), then
there exist $1\leq i_1<\ldots<i_n\leq k$ such that
$$
c_{i_1}\cdots c_{i_n}\det[u_{i_1},\ldots,u_{i_n}]^2\geq \binom{k }{ n}^{-1}.
$$
\end{lemma}

\section{Discrete isotropic measures, orthonormal bases and approximation by a regular simplex}

According to Lemma~3.2 and Lemma~5.1 in B\"or\"oczky, Hug \cite{BoH17}, the following auxiliary results are
available.

\begin{lemma}
\label{almostort}
Let  $v_1,\ldots,v_k\in \R^n\setminus \{0\}$
satisfy $\sum_{i=1}^k v_i \otimes v_i=\Id_n$, and let
$0<\eta<1/(3\sqrt{k})$. Assume for any $i\in\{1,\ldots,k\}$ that  $\|v_i\|\leq \eta$ or
there is some $j\in\{1,\ldots,n\}$ with $\angle(v_i,v_j)\leq \eta$.
Then there exists an orthonormal basis $w_1,\ldots,w_n$
such that $\angle(v_i,w_i)<3\sqrt{k}\,\eta$ for $i=1,\ldots,n$.
\end{lemma}

\begin{lemma}
\label{basis}
Let $e\in S^{n-1}$, and let $\tau\in (0,1/(2n))$. If $w_1,\ldots,w_n$ is an orthonormal basis
of $\R^n$ such that
$$\frac1{\sqrt{n}}-\tau<\langle e,w_i\rangle<\frac1{\sqrt{n}}+\tau\quad\text{
for \,$i=1,\ldots,n,$}
$$
then
there exists an orthonormal basis $\tilde{w}_1,\ldots,\tilde{w}_n$
such that $\langle e,\tilde{w}_i\rangle=\frac1{\sqrt{n}}$
and $\angle(w_i,\tilde{w}_i)<n\tau$ for $i=1,\ldots,n$.
\end{lemma}

Since $\sqrt{k}(n+1)<kn$ if $k>n\geq 2$, and $|\cos (\beta)-\frac1{\sqrt{n+1}}|\leq|\beta-\alpha|$ if
$\alpha=\arccos\frac1{\sqrt{n+1}}$, we deduce
from Lemmas~\ref{almostort} and \ref{basis} the following consequence.

\begin{coro}
\label{approxort}
Let  $k>n\geq 2$, let $\tilde{u}_1,\ldots,\tilde{u}_k,e\in S^n$ in $\R^{n+1}$ and $\tilde{c}_1,\ldots,\tilde{c}_k>0$
satisfy $\sum_{i=1}^k \tilde{c}_i \tilde{u}_i \otimes \tilde{u}_i=\Id_{n+1}$
and $\langle e,\tilde{u}_i\rangle=\frac1{\sqrt{n+1}}$ for $i=1,\ldots,k$, and let
$0<\eta<1/(6kn)$. Assume  for any $i\in\{1,\ldots,k\}$ that $\tilde{c}_i\leq \eta^2$ or there exists some $j\in\{1,\ldots,n+1\}$ with $\angle(\tilde{u}_i,\tilde{u}_j)\leq \eta$. Then there exists an orthonormal basis $\tilde{w}_1,\ldots,\tilde{w}_{n+1}$ of $\R^{n+1}$
such that $\langle e,\tilde{w}_i\rangle=\frac1{\sqrt{n+1}}$  and
$\angle(\tilde{u}_i,\tilde{w}_i)<3kn\eta$ for $i=1,\ldots,n+1$.
\end{coro}

For  $\tilde{w}_1,\ldots,\tilde{w}_{n+1}\in S^n$ with $\langle e,\tilde{w}_i\rangle=\frac1{\sqrt{n+1}}$  for $i=1,\ldots,n+1$,
the vectors
$\tilde{w}_1,\ldots,\tilde{w}_{n+1}$ form an orthonormal basis of $\R^{n+1}$ if and only if
their projection to $e^\bot$ form the vertices of a regular $n$-simplex. Therefore
Corollary~\ref{approxort} provides information on how close
$\text{conv}\{\tilde{u}_1,\ldots,\tilde{u}_{n+1}\}$ is to some regular $n$-simplex.
Lemma~5.2 in B\"or\"oczky, Hug \cite{BoH17} formulated this observation as follows.

\begin{lemma}
\label{closesimplex}
Let $Z$ be a polytope, and let $S$ be a regular simplex circumscribed to $B^n$.
Assume that the facets of $Z$ and $S$ touch $B^n$ at $u_1,\ldots,u_k$
and $w_1,\ldots,w_{n+1}$, respectively. Fix $\eta\in (0,1/(2n))$. If
for any $i\in\{1,\ldots,k\}$ there exists some $j\in\{1,\ldots,n+1\}$ such that $\angle (u_i,w_j)\le \eta$,
then
$$
(1-n \eta) S\subset  Z\subset (1+2n \eta) S.
$$
\end{lemma}

Finally, we need to estimate the difference of Gaussian measures of certain polytopes $Z\subset S$. Since in our case, $S\subset nB^n$, it is equivalent to estimate the volume difference
up to a factor depending on $n$. Our first estimate of this kind is Lemma~5.3 in B\"or\"oczky, Hug \cite{BoH17}.

\begin{lemma}
\label{closefarsimplex}
Let $Z$ be a polytope, and let $S$ be a regular simplex both circumscribed to $B^n$. Fix
$\alpha=9\cdot 2^{n+2}n^{2n+2}$ and $\eta\in (0,\alpha^{-1})$.
Assume that the facets of $Z$ and $S$ touch $B^n$ at $u_1,\ldots, u_k$, $k\geq n+1$,
and $w_1,\ldots,w_{n+1}$, respectively.  If $\angle (u_i,w_i)\le\eta$ for
$i=1,\ldots,n+1$  and $\angle (u_k,w_i)\geq \alpha\eta$ for  $i=1,\ldots,n+1$,
then
$$
V(Z)\leq \left(1-\frac{\min_{i=1,\ldots,n+1}\angle (u_k,w_i) }{2^{n+2}n^{2n}}\right)V(S).
$$
\end{lemma}

Secondly, we prove another estimate concerning the volume difference of a convex body and a simplex.

\begin{lemma}
\label{voldifffromsimplex}
Let $S$ be a regular simplex whose centroid is the origin, and
let $M_1\subset S$ and $M_2\supset S$ be convex bodies. Suppose that there is
some $\varepsilon\in(0,1)$ such that
$M_1\not\supset(1-\varepsilon)S$ for (i) and
 $(1+\varepsilon)S\not\supset M_2$ for (ii), respectively. Then
\begin{enumerate}
\item[{\rm (i)}] $V(S\setminus M_1)\geq \frac{n^n}{(n+1)^n}\ \varepsilon^n\  V(S)>
\frac{1}{e}\ \varepsilon^n\ V(S)$;
\item[{\rm (ii)}] $V(M_2\setminus S)\geq \frac{1}{n+1}\ \varepsilon\  V(S)$.
\end{enumerate}
\end{lemma}

\begin{proof} Let $R$ be the circumradius of $S$, let $v_1,\ldots,v_{n+1}$ be the vertices of $S$, and let $u_1,\ldots,u_{n+1}$ be the corresponding exterior unit normals of the facets, and hence
$$
v_i=-Ru_i\mbox{ \ for $i=1,\ldots,n+1$, and \ }S=\left\{x\in\R^n:
\mbox{$\langle x,u_i\rangle \leq \frac{R}n$
 \ for $i=1,\ldots,n+1$}\right\}.
$$

For (i), there exists a $v_i$ such that $(1-\varepsilon)v_i\not\in M_1$, and hence there exists a closed halfspace $H^+$
with $(1-\varepsilon)v_i\in H^+$ and $H^+\cap M_1=\emptyset$. We observe that $(1-\varepsilon)v_i$ is the centroid
of the simplex $S_\varepsilon=(1-\varepsilon)v_i+\varepsilon S\subset S$. Using
Gr\"unbaum's result \cite{Gru60} on minimal hyperplane sections of the simplex through its centroid, we obtain
$$
V(S\setminus M_1)> V(S_\varepsilon\cap H^+)\geq \frac{n^n}{(n+1)^n}\  V(S_\varepsilon)=
\frac{n^n}{(n+1)^n}\  \varepsilon^n\  V(S).
$$
For (ii), there exists an $x_0\in M_2\setminus((1+\varepsilon)S)$, and hence
there is a $u_j$ such that $\langle x_0,u_j\rangle >\frac{(1+\varepsilon)R}n$.
We write $F_j$ to denote the facet of $S$ with exterior unit normal $u_j$, and $|F_j|$ to denote
the $(n-1)$-volume of $F_j$. It follows that
$$
V(M_2\setminus S)\geq \frac1n\   \frac{\varepsilon\,R}n\  |F_j|=
\frac{\varepsilon}{n+1} (n+1)\frac{1}{n}\frac{R}{n}  |F_j|
=\frac{\varepsilon}{n+1}\  V(S), 
$$
which completes the proof.
\hfill\proofbox 
\end{proof}

\medskip

\noindent
{\bf Remark } The estimates in (i) and in (ii) are optimal.

\medskip

Finally, we provide some rough estimates {that} will be used repeatedly in the {sequel}.

\begin{lemma}\label{lemrough} 
Let $\Delta_n$ be a regular simplex inscribed {into} $B^n$, and let $\Delta_n^\circ=-n\Delta_n$ be its polar. Then
\begin{enumerate}
\item[{\rm (a)}] $\ell(\Delta_n)\le \sqrt{n}^3$, $\ell(\Delta_n^\circ)\le \sqrt{n}$,
\item[{\rm (b)}] $V(\Delta_2)\le 1.3$ and $V(\Delta_n)\le 1$ for $n\ge 3$,
\item[{\rm (c)}] $V(\Delta_n^\circ)=n^nV(\Delta_n)\ge (1+\frac{1}{n})^{\frac{n}{2}}> 1$,
\item[{\rm (d)}] $V(\Delta_n)\ge n^{-(n+2)}\ell (\Delta_n)$. 
\end{enumerate}
\end{lemma}

\begin{proof} (a) Since $\frac{1}{n}B^n\subset\Delta_n$ and by an application of \cite[(7)]{Wendel48}, we get
\begin{align*}
\ell(\Delta_n)\le n\ell(B^n)=n\int_{\R^n}\|x\|\, \gamma_n(dx)=\frac{n^2}{\sqrt{2}}\frac{\Gamma\left(\frac{n+1}{2}\right)}{\Gamma\left(\frac{n+2}{2}\right)}  
\le \sqrt{n}^3.
\end{align*}

For (b) and (c), we have
$$
\frac{1}{n^n}\le V(\Delta_n)=\left(1+\frac{1}{n}\right)^{\frac{n}{2}}\frac{\sqrt{n+1}}{n!}\le {\frac{\sqrt e\,\sqrt{n+1} }{n!}}<1,
$$
where the  upper bound on the right side only holds for $n\ge 3$. 

(d) follows from (a) and (c). \hfill \proofbox
\end{proof}

\section{On the derivatives of the transportation map}\label{sec6}

Let $f$ and $h$ be probability density functions on $\R$ that are continuous and differentiable on the interiors of their supports,
which are assumed to be intervals   $I_f,I_h\subset\R$.   Then there exists a transportation map
$T:I_f\to I_h$ determined by
$$
\int_{-\infty}^x f=\int_{-\infty}^{T(x)} h.
$$
For $x\in I_f$ it follows that
\begin{align}
\label{Tder}
T'(x)&=\frac{f(x)}{h(T(x))},\\
\label{Tsecondder}
T''(x)&=\frac{f(x)^2}{h(T(x))}\left(\frac{f'(x)}{f(x)^2}-\frac{h'(T(x))}{h(T(x))^2}\right).
\end{align}

Let $g$ be the standard Gaussian density $g(t)=\sqrt{{2\pi}}^{-1}\,e^{-t^2/2}$, $t\in\R$, and
for $s\in \R$, let $g_s$ be the truncated Gaussian density
$$
g_s(x)=
\begin{cases}
\displaystyle{\left( \int_0^\infty g(t-s)\,dt\right)^{-1}g(x-s)},& \mbox{ if }x\geq 0,\\[1.5ex]
0,& \mbox{ if } x< 0.
\end{cases}
$$

We frequently use that if $s\geq 0$, then
\begin{equation*}
\frac12\leq \int_0^\infty g(t-s)\,dt<1.
\end{equation*}

We are going to apply \eqref{Tsecondder} either in the case when $h=g$ and $f=g_s$, for some
$s\in \R$, or when the roles of $f,g$ are reversed. In particular, we consider the transport maps
$\varphi_s:\,(0,\infty)\to\R$ and $\psi_s:\,\R\to(0,\infty)$ such that
$$
\int_0^{x}g_s=\int_{-\infty}^{\varphi_s(x)}g\quad \text{ and} \quad
\int_{-\infty}^{y}g=\int_0^{\psi_s(y)}g_s.
$$
Clearly, $\varphi_s$ and $\psi_s$ are inverses of each other for any given $s\in\R$.

\begin{lemma}
\label{gstransporterror}
Let $s\in[0,0.15]$.
\begin{enumerate}
\item[{\rm (i)}] If $x\in[0.74,0.77]$, then
$0<\varphi_s(x)<  0.16$,
$1.3\leq\varphi'_s(x)\leq 2.05$ and $\varphi_s''(x)\leq -0.25$.
\item[{\rm(ii)}] If $y\in[0,0.15]$, then $0<\psi_s(y)<  0.85$,
$0.49\leq\psi'_s(y)\leq 0.77$ and $\psi''_s(y)\geq 0.07$.
\end{enumerate}
\end{lemma}
\begin{proof}
We define $\alpha,\beta,\gamma,\delta,\xi>0$ via
\begin{align*}
\int_{\delta}^{\infty}g&=\frac7{32},\mbox{ \ and hence $0.77<\delta<0.78$},\\
\int_{\xi}^{\infty}g&=\frac{63}{256},\mbox{ \ and hence $0.68<\xi<0.69$},\\
\int_{\alpha}^{\infty}g&=\frac14,\mbox{ \ and hence $0.67<\alpha<0.68$},\\
\int_{\beta}^{\infty}g&=\frac9{32},\mbox{ \ and hence $0.57<\beta<0.58$},\\
\int_{\gamma}^{\infty}g&=\frac7{16},\mbox{ \ and hence $0.15<\gamma<0.16$},
\end{align*}
and therefore
\begin{align}
\nonumber
\psi_0(0)&=\alpha,\\
\label{deltalarge}
\psi_0(\gamma)&=\delta>0.77,\\
\label{gammabetasmall}
\psi_\gamma(0)&=\gamma+\beta<0.74,\\
\label{gammagammasmall}
\psi_\gamma(\gamma)&=\gamma+\xi<0.85.
\end{align}

First, we show that if $y\ge 0$, then the map $s\mapsto \psi_s(y)-s$, $s\ge 0$, is strictly decreasing
and $\psi_s(y)-s>0$.

In fact, by definition we have
$$
\int_{-\infty}^y g=\int_0^{\psi_s(y)}g_s=\left(\int_{-s}^\infty g\right)^{-1}\int_{-s}^{\psi_s(y)-s}g,
$$
and hence
$$
\int_{-\infty}^y g \int_{-s}^\infty g =\int_{-s}^\infty g-\int_{\psi_s(y)-s}^\infty g
$$
or
\begin{equation}\label{eqsign}
\int_{-s}^\infty g \int_y^\infty g=\int_{\psi_s(y)-s}^\infty g,
\end{equation}
which implies the monotonicity statement. Moreover, since the left side of \eqref{eqsign} is less than $1/2$, it
follows that $\psi_s(y)-s>0$ for $y\ge 0$.

\medskip

Now, we show that if $y\in[0,\gamma]$, then
\begin{equation}
\label{psimonofs}
\mbox{$\psi_s(y)$ is a monotone increasing function of $s\geq 0$.}
\end{equation}

For the proof, we show that if $0\le s<s'$, then the inequality
\begin{equation}\label{ineqref1}
\int_0^{\psi_s(y)}g_{s'}\le \int_0^{\psi_s(y)}g_{s}=\int_{-\infty}^y  g
\end{equation}
holds. We set $x:=\psi_s(y)$, $\Delta:=s'-s\ge 0$ and define
$$
A:=\int_0^xg(\sigma-s)\, d\sigma,\qquad B:=\int_0^\infty g(\sigma-s)\, d\sigma
$$
and
$$
a:=\int_{-\Delta}^0 g(\sigma-s)\, d\sigma,\qquad b:=\int_{x-\Delta}^x g(\sigma-s)\, d\sigma.
$$
Note that
$$
\int_0^{x}g_{s'}=\frac{\int_{-\Delta}^{x-\Delta}g(\tau-s)\, d\tau}{\int_{-\Delta}^{\infty}g(\tau-s)\, d\tau}
=\frac{a+A-b}{a+B}
$$
and the right-hand side of \eqref{ineqref1} equals $A/B$. Hence \eqref{ineqref1} is equivalent to
$$
\frac{a+A-b}{a+B}\le \frac{A}{B}\qquad\text{or}\qquad \frac{a}{b}\le \frac{1}{1-\frac{A}{B}}.
$$
Since
$$
\frac{A}{B}=\int_{-\infty}^y  g\ge \frac{1}{2},
$$
 it is sufficient to show that $a/b\le 2$.

By the symmetry of $g$, translation invariance of Lebesgue measure and inserting again $\Delta=s'-s$ and $x=\psi_s(y)$, we get
$$
a=\int_s^{s'}g,\qquad b=\int_{s-\psi_s(y)}^{s'-\psi_s(y)}g.
$$
Thus it remains to be shown that
\begin{equation}\label{ineqref2}
\int_s^{s'}e^{-t^2/2}\, dt\le \int_{s-\psi_s(y)}^{s'-\psi_s(y)} 2e^{-t^2/2}\, dt
\end{equation}
for $0\le s<s'$ and $y\in [0,\gamma]$. To see this, we distinguish two cases.

If $s'-\psi_s(y)\le 0$, then $2e^{-t^2/2}\ge 1$ for $t\in [s-\psi_s(y),s'-\psi_s(y)]\subset(-\infty,0]$, since
$$
\psi_s(y)-s\le \psi_s(\gamma)-s\le \psi_0(\gamma)-0=\psi_0(\gamma)<0.78
$$
and $2\exp(-0.5\cdot 0.78^2)\ge 1.4>1$. Since $e^{-t^2/2}\le 1$ for $t\in [ s,s']$, the assertion follows in this case.

If $s'-\psi_s(y)> 0$, then by the previous reasoning and since $s-\psi_s(y)<0$, we have
\begin{equation}\label{ineqref3}
\int_s^{\psi_s(y)}e^{-t^2/2}\, dt\le \int_{s-\psi_s(y)}^{0}2e^{-t^2/2}\, dt,
\end{equation}
and since $t\mapsto e^{-t^2/2}$, $t\ge 0$, is decreasing, we have
\begin{equation}\label{ineqref4}
\int_{\psi_s(y)}^{s'}e^{-t^2/2}\, dt\le \int_{0}^{s'-\psi_s(y)}e^{-t^2/2}\, dt,
\end{equation}
so that \eqref{ineqref3} and \eqref{ineqref4} again imply \eqref{ineqref2}.


\medskip

We deduce from \eqref{deltalarge}, \eqref{gammabetasmall} and \eqref{psimonofs} that
$\psi_s(0)\le \psi_s(\gamma)<0.74$, $\psi_s(\gamma)\ge \psi_0(\gamma)>0.77$, and hence
\begin{equation}
\label{psirange}
\mbox{$[0.74,0.77]\subset \psi_s((0,\gamma))\quad$ if $s\in[0,\gamma]$.}
\end{equation}

We note that if $y\in[0,\gamma]$, then
\begin{equation}
\label{seconderg}
\frac{g'(y)}{g(y)^2}=-\sqrt{2\pi}\, ye^{y^2/2}\geq -\sqrt{2\pi}\cdot 0.17.
\end{equation}
On the other hand, if $0\leq s\leq \gamma$ and $y\geq 0$, then
$$
\frac12\leq \int_{-s}^\infty g\quad \mbox{ and }\quad
\psi_s(y)-s\geq  \psi_\gamma(y)-\gamma\geq \psi_\gamma(0)-\gamma=\beta> 0.57,
$$
and therefore
\begin{align}
\label{secondergs}
\frac{g'_s(\psi_s(y))}{g_s(\psi_s(y))^2}&=
-\sqrt{2\pi}\left(\int_{-s}^\infty g\right)(\psi_s(y)-s)e^{\frac{(\psi_s(y)-s)^2}{2}}\nonumber\\
&\leq -\frac{\sqrt{2\pi}}{2}\beta\, e^{\frac{\beta^2}2}<-\sqrt{2\pi}\cdot 0.33.
\end{align}
Combining \eqref{seconderg} and \eqref{secondergs}, for $s,y\in[0,\gamma]$ we get
\begin{equation}
\label{seconderggs}
\frac{g'(y)}{g(y)^2}-\frac{g'_s(\psi_s(y))}{g_s(\psi_s(y))^2}\geq \sqrt{2\pi}\cdot 0.15.
\end{equation}
If $s,y\in[0,\gamma]$, then
\begin{equation}\label{efficient}
g_s(\psi_s(y))\leq\frac{2}{\sqrt{2\pi}} \quad\text{
and }\quad g(y)\geq \frac{e^{-\gamma^2/2}}{\sqrt{2\pi}}>\frac{0.98}{\sqrt{2\pi}}.
\end{equation}
Hence, for $s,y\in[0,\gamma]$  we deduce
from \eqref{Tsecondder}, \eqref{seconderggs} and \eqref{efficient} that
\begin{equation}
\label{seconderpsis}
\psi''_s(y)\geq  \frac{0.98^2}{2}\cdot 0.15>0.07.
\end{equation}
In addition, \eqref{gammagammasmall} and \eqref{psimonofs} imply that  if $s,y\in[0,\gamma]$, then
\begin{equation}
\label{functionupperpsis}
\psi_s(y)<  0.85.
\end{equation}
To estimate the first derivative $\psi'_s$, we use that \eqref{Tder} yields
\begin{equation}
\label{firstderpsis}
\psi'_s(y)=\frac{g(y)}{g_s(\psi_s(y))}.
\end{equation}
If $s,y\in[0,\gamma]$, then \eqref{efficient}
and \eqref{firstderpsis} yield
\begin{equation}
\label{firstderlowerpsis}
\psi'_s(y)\geq \frac{0.98/\sqrt{2\pi}}{2/\sqrt{2\pi}}=0.49.
\end{equation}
On the other hand, if $s,y\in[0,\gamma]$, then $0<\psi_s(y)-s\le \psi_0(y)-0\le\psi_0(\gamma)=\delta<0.78$, and hence
\begin{equation}\label{eqeffref1}
g_s(\psi_s(y))=\frac{1}{\sqrt{2\pi}}\frac{e^{-(\psi_s(y)-s)^2/2}}{\int_{-s}^\infty g}\ge \frac{1}{\sqrt{2\pi}}
\frac{e^{-0.78^2/2}}{\int_{-0.16}^\infty g}\ge \frac{1}{\sqrt{2\pi}}\cdot 1.3.
\end{equation}
Hence we deduce
from \eqref{firstderpsis} that
\begin{equation}
\label{firstderupperpsis}
\psi'_s(y)\leq \frac{1/\sqrt{2\pi}}{1.3/\sqrt{2\pi}}<0.77.
\end{equation}
We conclude (ii) from \eqref{seconderpsis}, \eqref{functionupperpsis}, \eqref{firstderlowerpsis}
and \eqref{firstderupperpsis}.

Turning to $\varphi_s$, \eqref{psirange} yields
\begin{equation}
\label{phirange}
\mbox{$\varphi_s([0.74,0.77])\subset (0,\gamma)\quad $ if $s\in[0,\gamma]$.}
\end{equation}
It follows from \eqref{seconderggs} and \eqref{phirange}
that if $s\in[0,\gamma]$ and $x\in [0.74,0.77]$, then
\begin{equation}
\label{phiseconderggs}
\frac{g'_s(x)}{g_s(x)^2}-\frac{g'(\varphi_s(x)))}{g(\varphi_s(x))^2}\leq -\sqrt{2\pi}\cdot 0.15.
\end{equation}
Now if $s\in[0,\gamma]$ and $x\in [0.74,0.77]$, then we have
\begin{equation}\label{efficient2}
g_s(x)\ge \frac{\frac{1}{\sqrt{2\pi}}e^{-0.77^2/2}}{\int_{-0.16}^\infty g}>\frac{1.3}{\sqrt{2\pi}},\qquad
g(\varphi_s(x))<\frac{1}{\sqrt{2\pi}}.
\end{equation}
Hence, \eqref{efficient2},
  \eqref{Tsecondder} and \eqref{phiseconderggs} imply that
\begin{equation}
\label{seconderphis}
\varphi''_s(x)\leq  -1.3^2\cdot 0.15<-0.25.
\end{equation}
To estimate the first derivative $\varphi'_s$, we use that \eqref{Tder} yields
\begin{equation}
\label{firstderphis}
\varphi'_s(x)=\frac{g_s(x)}{g(\varphi_s(x))}.
\end{equation}

If $s\in[0,\gamma]$ and $x\in [0.74,0.77]$, then we conclude from  \eqref{phirange} that
$$g(\varphi_s(x))\geq \frac{e^{-\gamma^2/2}}{\sqrt{2\pi}}>\frac{0.98}{\sqrt{2\pi}}\quad\text{ and }\quad g_s(x)\leq\frac{2}{\sqrt{2\pi}},
$$
and hence \eqref{firstderphis} implies that
\begin{equation}
\label{firstderupperphis}
\varphi'_s(x)\leq \frac{2/\sqrt{2\pi}}{0.98/\sqrt{2\pi}}<2.05.
\end{equation}

On the other hand, if $s\in[0,\gamma]$ and $x\in [0.74,0.77]$, then
we deduce from \eqref{efficient2} and \eqref{firstderphis} that
\begin{equation}
\label{firstderlowerphis}
\varphi'_s(x)> \frac{1.3/\sqrt{2\pi}}{1/\sqrt{2\pi}}=1.3.
\end{equation}
We conclude (ii) from \eqref{seconderphis}, \eqref{phirange}, \eqref{firstderupperphis}
and \eqref{firstderlowerphis}.
\hfill \proofbox
\end{proof}

\medskip 

In Proposition~\ref{BLRBLgsg}, we use the following notation. We fix an $e\in S^n\subset\R^{n+1}$, and
identify $e^\bot\subset \R^{n+1}$ with $\R^n$.
For $k\geq n+1$, let $u_1,\ldots,u_k\in S^{n-1}$ and $c_1,\ldots,c_k>0$ be such that
\begin{equation}
\label{uiciinRn}
\begin{array}{rcl}
\sum_{i=1}^kc_iu_i\otimes u_i&=&\Id_n,\\[1ex]
\sum_{i=1}^kc_iu_i&=&o.
\end{array}
\end{equation}
For each $u_i$, we consider
$$
\begin{array}{rcl}
\tilde{u}_i&=&\frac{\sqrt{n}}{\sqrt{n+1}}\,u_i+\frac{1}{\sqrt{n+1}}\,e\in S^n,\\[1ex]
\tilde{c}_i&=&\frac{n+1}n\,c_i,
\end{array}
$$
and hence \eqref{uiciinRn} yields that
\begin{equation*}
\sum_{i=1}^k\tilde{c}_i\,\tilde{u}_i\otimes \tilde{u}_i=\Id_{n+1}.
\end{equation*}

\begin{prop}
\label{BLRBLgsg}
With the above notation, let $k\leq 2n^2$, let $s\in[0,0.15]$
and let $\varepsilon\in(0,n^{-56n})$.
If
\begin{equation}
\label{BLgsg}
\int_{\R^{n+1}}\prod_{i=1}^kg_s(\langle x,\tilde{u}_i\rangle)^{\tilde{c}_i}\,dx\geq
1-\varepsilon, \text{ or}
\end{equation}
\begin{equation}
\label{RBLgsg}
{\int^*_{\R^{n+1}}}\sup_{x=\sum_{i=1}^k\tilde{c}_i\theta_i\tilde{u}_i}\prod_{i=1}^kg_s(\theta_i)^{\tilde{c}_i}\,dx
\leq  1+\varepsilon,
\end{equation}
then there exists a regular simplex with vertices $w_1,\ldots,w_{n+1}\in S^{n-1}$
and $i_1<\ldots<i_{n+1}$ such that
$\angle(u_{i_j},w_j)<n^{14n}\varepsilon^{1/4}$ for $j=1,\ldots,n+1$.
\end{prop}

\begin{proof}
According to Lemma~\ref{ciuibig}, we may assume
\begin{equation}
\label{prod1n+1}
\tilde{c}_{1} \cdots \tilde{c}_{n+1}
\det[\tilde{u}_{1},\ldots,\tilde{u}_{n+1}]^2\geq \binom{k }{ n+1}^{-1}.
\end{equation}
For $\eta=n^{10n}\varepsilon^{1/4}<1$, we claim that if $i\in\{1,\ldots,k\}$, then
\begin{equation}
\label{uconsitionBLRBLgsg}
\mbox{$\tilde{c}_i\leq \eta^2$, or there exists some $j\in\{1,\ldots,n+1\}$ with
$\angle(\tilde{u}_i,\tilde{u}_j)\leq \eta$.}
\end{equation}
We suppose that \eqref{uconsitionBLRBLgsg} does not hold, hence we may assume
$$
\tilde{c}_{n+2}>\eta^2\mbox{ \ and \ }\angle(\tilde{u}_i,\tilde{u}_{n+2})> \eta
\mbox{ \ for $i=1,\ldots,n+1$.}
$$
We can write $\tilde{u}_{n+2}=\sum_{i=1}^{n+1}\lambda_i\tilde{u}_i$, where $\lambda_1,\ldots,\lambda_{n+1}\in\R$ are uniquely determined and satisfy $\lambda_1+\cdots+\lambda_{n+1}=1$. Hence we may assume that $\lambda_1\geq \frac1{n+1}$.
Therefore $\tilde{c}_{n+2}>\eta^2$, $\tilde{c}_{1}\leq 1$ and \eqref{prod1n+1} imply
$$
\tilde{c}_2\cdots \tilde{c}_{n+2}\det[\tilde{u}_2,\ldots,\tilde{u}_{n+2}]^2\geq
\binom{k }{ n+1}^{-1}\,\frac{\eta^2}{(n+1)^2}\geq
\frac{(n+1)!}{(2n^2)^{n+1}}\,\frac{\eta^2}{(n+1)^2}.
$$
Here $\frac{(n+1)!}{(n+1)^2}\geq \frac{n^n}{(n+1)e^n}>\frac{n^{n-1}}{3^{n+1}}$, and thus
\begin{equation}
\label{prod2n+2}
\tilde{c}_2\cdots \tilde{c}_{n+2}\det[\tilde{u}_2,\ldots,\tilde{u}_{n+2}]^2\geq
\frac{\eta^2}{3^{n+1}2^{n+1}n^{n+3}}>\frac{\eta^2}{n^{4n+6}}.
\end{equation}
In addition, $\angle(\tilde{u}_1,\tilde{u}_{n+2})> \eta$ yields
\begin{equation}
\label{diff1n+2}
\|\tilde{u}_1-\tilde{u}_{n+2}\|>\eta/2.
\end{equation}

We prove \eqref{uconsitionBLRBLgsg} separately for \eqref{BLgsg} and \eqref{RBLgsg}.

We start with the Brascamp-Lieb inequality; namely, we assume that \eqref{BLgsg} holds. We observe that
if $i=1,\ldots,k$, then
\begin{equation*}
0.74<\langle x,\tilde{u}_i\rangle< 0.77\mbox{ \ \ \ for
$x\in 0.755\sqrt{n+1}\,e+ 0.01B^n$,}
\end{equation*}
and define
$$
\Xi:=0.755\sqrt{n+1}\,e+0.005\,\frac{\tilde{u}_1-\tilde{u}_{n+2}}{\|\tilde{u}_1-\tilde{u}_{n+2}\|}
+0.001B^n\subset 0.755\sqrt{n+1}\,e+ 0.01B^n.
$$
It follows using also \eqref{diff1n+2}, $\langle e,\tilde{u}_1-\tilde{u}_{n+2}\rangle=0$
and
 $V(B^n)=\frac{\pi^{\frac{n}2}}{\Gamma(\frac{n}2+1)}>\frac{(2e\pi)^{\frac{n}2}}{n^{n/2}\sqrt{2\pi}e^{1/(6n)}}>
\frac1{n^{n/2}}$
 that
\begin{eqnarray}
\label{XiBL1}
\langle x,\tilde{u}_1\rangle,\ldots,\langle x,\tilde{u}_k\rangle&\in&[0.74,0.77]
\mbox{ \ for $x\in \Xi$},\\
\label{XiBL2}
\langle x,\tilde{u}_1\rangle-\langle x,\tilde{u}_{n+2}\rangle&\geq&0.002\eta>2^{-9}\eta
\mbox{ \ for $x\in \Xi$},\\
\label{XiBL3}
\Xi&\subset&{\cal C}:=\{x\in\R^{n+1}:\, \langle \tilde{u}_i,x\rangle>0\;\forall i=1,\ldots,k\},\\
\label{XiBL4}
V(\Xi)&=&0.001^nV(B^n)>\frac1{n^{11n}},
\end{eqnarray}
where \eqref{XiBL3} is a consequence of \eqref{XiBL1}.
In addition,
we consider the map $\Theta:{\cal C}\to\R^{n+1}$ with
$$
\Theta(x)=\sum_{i=1}^k\tilde{c}_i \varphi_s(\langle \tilde{u}_i,x\rangle )\,\tilde{u}_i,\qquad x\in {\cal C},
$$
which satisfies
$$
d\Theta(x)=\sum_{i=1}^k\tilde{c}_i\varphi'_s(\langle \tilde{u}_i,x\rangle )\,\tilde{u}_i\otimes \tilde{u}_i.
$$
As we have seen, $d\Theta$ is positive definite and $\Theta:{\cal C}\to\R^{n+1}$ is
 injective (see \cite{Bar97,Bar98}).
Therefore, applying first \eqref{BLgsg}, then \eqref{masstrans}, and after that the definition of $\Theta$ and
\eqref{BLRBLquad}, we obtain
\begin{align}
\nonumber
1-\varepsilon&\leq \int_{\R^{n+1}}\prod_{i=1}^kg_s(\langle x,\tilde{u}_i\rangle)^{\tilde{c}_i}\,dx=
\int_{{\cal C}}\prod_{i=1}^kg_s(\langle x,\tilde{u}_i\rangle)^{\tilde{c}_i}\,dx\\
\nonumber
&\leq
\int_{{\cal C}}\left(\prod_{i=1}^kg(\varphi_s(\langle x,\tilde{u}_i\rangle))^{\tilde{c}_i}\right)
\left(\prod_{i=1}^k\varphi'_s(\langle x,\tilde{u}_i\rangle)^{\tilde{c}_i}\right)\,dx\\
\nonumber
&= \left(\frac{1}{2\pi}\right)^{\frac{n+1}2}\int_{{\cal C}}\left(\prod_{i=1}^k
e^{-\tilde{c}_i\varphi_s(\langle x,\tilde{u}_i\rangle)^2/2}\right)
\left(\prod_{i=1}^k\varphi'_s(\langle x,\tilde{u}_i\rangle)^{\tilde{c}_i}\right)\,dx\\
\label{BLstabintermediate}
&\leq \left(\frac{1}{2\pi}\right)^{\frac{n+1}2}\int_{{\cal C}}
e^{-\|\Theta(x)\|^2/2}
\left(\prod_{i=1}^k\varphi'_s(\langle x,\tilde{u}_i\rangle)^{\tilde{c}_i}\right)\,dx.
\end{align}
We deduce from \eqref{BallBarthe} that
\begin{equation}
\label{BallBartheBL}
\prod_{i=1}^k\varphi'_s(\langle x,\tilde{u}_i\rangle)^{\tilde{c}_i}\leq
\det\left(\sum_{i=1}^k\tilde{c}_i\varphi'_s(\langle x,\tilde{u}_i\rangle)\tilde{u}_i\otimes\tilde{u}_i\right)
=\det (d\Theta(x))
\end{equation}
for any $x\in{\cal C}$.

If $s\in [0,0.15]$ and $x\in \Xi$, then we can improve \eqref{BallBartheBL} using
Corollary~\ref{Ball-Barthe-stab}, based on \eqref{prod1n+1} and \eqref{prod2n+2} with
$$
\beta_0= \frac{\eta^2}{n^{4n+6}}.
$$
Hence, applying first Corollary~\ref{Ball-Barthe-stab}, then
Lemma~\ref{gstransporterror} (i), \eqref{XiBL1}, \eqref{XiBL2} and finally $\eta<1$, we get
\begin{align*}
\prod_{i=1}^k\varphi'_s(\langle x,\tilde{u}_i\rangle)^{\tilde{c}_i}&\leq \left(1+
\frac{\beta_0(\varphi'_s(\langle x,\tilde{u}_1\rangle)-\varphi'_s(\langle x,\tilde{u}_{n+2}\rangle))^2}
{4(\varphi'_s(\langle x,\tilde{u}_1\rangle)+\varphi'_s(\langle x,\tilde{u}_{n+2}\rangle))^2}\right)^{-1}
\det\left( d\Theta(x)\right) \\
&\leq\left(1+
\frac{\beta_0\left(0.25(\langle x,\tilde{u}_1\rangle-\langle x,\tilde{u}_{n+2}\rangle)\right)^2}
{4(2\cdot 2.05)^2}\right)^{-1}
\det\left( d\Theta(x)\right)\\
&\leq\left(1+\frac{\eta^40.25^2 2^{-18}}
{n^{4n+6}16\cdot 2.05^2}\right)^{-1}
\det\left( d\Theta(x)\right)\\
&\leq  \left(1+\frac{\eta^4}
{n^{4n+35}}\right)^{-1}
\det\left( d\Theta(x)\right)
\leq
\left(1-\frac{\eta^4}{n^{4n+36}}\right) \det\left( d\Theta(x)\right).
\end{align*}
Moreover, if $s\in [0,0.15]$ and $x\in \Xi$ we deduce from \eqref{BallBartheBL} and Lemma~\ref{gstransporterror} (i)  that
$$
\det\left( d\Theta(x)\right)\geq \prod_{i=1}^k\varphi'_s(\langle x,\tilde{u}_i\rangle)^{\tilde{c}_i}\geq
\prod_{i=1}^k 1^{\tilde{c}_i}=1.
$$
Thus if $s\in [0,0.15]$ and $x\in \Xi$, then
\begin{equation}
\label{XiprodBL}
\prod_{i=1}^k\varphi'_s(\langle x,\tilde{u}_i\rangle)^{\tilde{c}_i}\leq
\det\left( d\Theta(x)\right)-\frac{\eta^4}{n^{4n+36}}.
\end{equation}
In addition, if $s\in [0,0.15]$ and $x\in \Xi$, then $\langle x,\tilde u_i\rangle \in [0.74,0.77]$ by \eqref{XiBL1}
and hence $\varphi_s(\langle x,\tilde u_i\rangle)\subset (0,\gamma)$ by \eqref{phirange}. Therefore,
the definition of $\Theta(x)$, \eqref{BLRBLquad}  and \eqref{cisum} imply
$$
\|\Theta(x)\|^2\leq \sum_{i=1}^k\tilde{c}_i \varphi_s(\langle \tilde{u}_i,x\rangle )^2\leq
\sum_{i=1}^k\tilde{c}_i 0.16^2=0.16^2(n+1),
$$
and hence
$$
\left(\frac{1}{2\pi}\right)^{\frac{n+1}2}e^{-\|\Theta(x)\|^2/2}\geq
\left(\frac{1}{2\pi}\right)^{\frac{n+1}2}e^{-(n+1)0.16^2/2}>n^{-n-3}.
$$
Applying first \eqref{XiBL3}, using \eqref{BallBartheBL} and \eqref{XiprodBL} in \eqref{BLstabintermediate},
 then the substitution $z=\Theta(x)$,  and finally also \eqref{XiBL4}, we get
\begin{align*}
1-\varepsilon&\leq \left(\frac{1}{2\pi}\right)^{\frac{n+1}2}\int_{{\cal C}}
e^{-\|\Theta(x)\|^2/2}\det\left( d\Theta(x)\right)\,dx
-\int_\Xi \frac{\eta^4}{n^{5n+39}}\,dx\\
&\leq  \left(\frac{1}{2\pi}\right)^{\frac{n+1}2}\int_{\R^{n+1}}e^{-\|z\|^2/2}\,dz
-\frac{1}{n^{11n}}\, \frac{\eta^4}{n^{5n+39}}\;\leq \;1-\frac{\eta^4}{n^{39n}}.
\end{align*}
This contradicts $\eta=n^{10n}\varepsilon^{1/4}$, and hence we conclude \eqref{uconsitionBLRBLgsg}
in the case of the Brascamp-Lieb inequality.

\bigskip

Now we consider the {reverse} Brascamp-Lieb inequality; namely, we assume that \eqref{RBLgsg} holds. We observe that
if $i\in\{1,\ldots,k\}$, then
\begin{equation*}
0\leq\langle x,\tilde{u}_i\rangle\leq 0.15\mbox{ \ \ \ for
$x\in 0.1\sqrt{n+1}\,e+ 0.05B^n$,}
\end{equation*}
and define
$$
\widetilde{\Xi}:=0.1\sqrt{n+1}\,e+0.03\,\frac{\tilde{u}_1-\tilde{u}_{n+2}}{\|\tilde{u}_1-\tilde{u}_{n+2}\|}
+0.01B^n\subset 0.1\sqrt{n+1}\,e+ 0.05B^n.
$$
It follows using again \eqref{diff1n+2}, $\langle e,\tilde{u}_1-\tilde{u}_{n+2}\rangle=0$
and
 $V(B^n)>\frac1{n^{n/2}}$
 that
\begin{eqnarray}
\label{XiRBL1}
\langle y,\tilde{u}_1\rangle,\ldots,\langle y,\tilde{u}_k\rangle&\in&[0,0.15]
\mbox{ \ for $y\in \widetilde{\Xi}$}\\
\label{XiRBL2}
\langle y,\tilde{u}_1\rangle-\langle y,\tilde{u}_{n+2}\rangle&\geq&0.01\eta>2^{-7}\eta
\mbox{ \ for $y\in \widetilde{\Xi}$}\\
\label{XiRBL3}
V(\widetilde{\Xi})&=&0.01^nV(B^n)>\frac1{n^{8n}}.
\end{eqnarray}

In addition,
we consider the map $\Psi:\R^{n+1}\to\R^{n+1}$ with
$$
\Psi(y)=\sum_{i=1}^k\tilde{c}_i \psi_s(\langle \tilde{u}_i,y\rangle )\,\tilde{u}_i,
$$
which satisfies
$$
d\Psi(y)=\sum_{i=1}^k\tilde{c}_i\psi'_s(\langle \tilde{u}_i,y\rangle )\,\tilde{u}_i\otimes \tilde{u}_i.
$$
As we have seen, $d\Psi$ is positive definite and $\Psi:\R^{n+1}\to \R^{n+1}$ is
 injective (see \cite{Bar97,Bar98}).
Therefore, applying first \eqref{RBLgsg}, then the definition of $\Psi$  imply
\begin{align}
\nonumber
1+\varepsilon&\geq
	{\int^*_{\R^{n+1}}}\sup_{x=\sum_{i=1}^k\tilde{c}_i\theta_i\tilde{u}_i}
\prod_{i=1}^kg_s(\theta_i)^{\tilde{c}_i}\,dx\\
\nonumber
&\geq
{\int^*_{\R^{n+1}}}\left(\sup_{\Psi(y)=\sum_{i=1}^k\tilde{c}_i\theta_i\tilde{u}_i}\prod_{i=1}^k
g_s(\theta_i)^{\tilde{c}_i}\right)
\det\left( d\Psi(y)\right)\,dy\\
\label{RBLstabintermediate}
&\geq  \int_{\R^{n+1}}\left(\prod_{i=1}^kg_s(\psi_s(\langle \tilde{u}_i,y\rangle))^{\tilde{c}_i} \right)
\det\left(\sum_{i=1}^k\tilde{c}_i\psi'_s(\langle \tilde{u}_i,y\rangle )\,\tilde{u}_i\otimes \tilde{u}_i\right)\,dy.
\end{align}
Using  \eqref{BallBarthe}, for  $y\in\R^{n+1}$  we can bound the determinant in \eqref{RBLstabintermediate} from below by
\begin{equation}
\label{BallBartheRBL}
\det\left(\sum_{i=1}^k\tilde{c}_i\psi'_s(\langle y,\tilde{u}_i\rangle)\tilde{u}_i\otimes\tilde{u}_i\right)
\geq \prod_{i=1}^k\psi'_s(\langle y,\tilde{u}_i\rangle)^{\tilde{c}_i}.
\end{equation}
If $s\in [0,0.15]$ and $y\in \widetilde{\Xi}$, an application of Corollary~\ref{Ball-Barthe-stab} with
$\beta_0=  \eta^2/{n^{4n+6}}$,  Lemma~\ref{gstransporterror} (ii), \eqref{XiRBL1} and \eqref{XiRBL2}, allow us to improve \eqref{BallBartheRBL} to get
\begin{align*}
\det\left(\sum_{i=1}^k\tilde{c}_i\psi'_s(\langle y,\tilde{u}_i\rangle)\tilde{u}_i\otimes\tilde{u}_i\right)
&\geq \left(1+
\frac{\beta_0(\psi'_s(\langle y,\tilde{u}_1\rangle)-\psi'_s(\langle y,\tilde{u}_{n+2}\rangle))^2}
{4(\psi'_s(\langle y,\tilde{u}_1\rangle)+\psi'_s(\langle y,\tilde{u}_{n+2}\rangle))^2}\right)
\prod_{i=1}^k\psi'_s(\langle y,\tilde{u}_i\rangle)^{\tilde{c}_i} \\
&\geq\left(1+
\frac{\beta_0\left(0.07(\langle y,\tilde{u}_1\rangle-\langle y,\tilde{u}_{n+2}\rangle)\right)^2}
{4(2\cdot 0.77)^2}\right)
\prod_{i=1}^k\psi'_s(\langle y,\tilde{u}_i\rangle)^{\tilde{c}_i}\\
&\geq\left(1+\frac{\eta^40.07^2 2^{-14}}
{n^{4n+6}16\cdot 0.77^2}\right)
\prod_{i=1}^k\psi'_s(\langle y,\tilde{u}_i\rangle)^{\tilde{c}_i}\\
&\geq  \left(1+\frac{\eta^4}
{n^{4n+31}}\right)
\prod_{i=1}^k\psi'_s(\langle y,\tilde{u}_i\rangle)^{\tilde{c}_i}.
\end{align*}
Moreover, if $s\in [0,0.15]$ and $y\in \widetilde{\Xi}$ we deduce from  Lemma~\ref{gstransporterror} (ii) and \eqref{cisum} that
$$
\prod_{i=1}^k\psi'_s(\langle y,\tilde{u}_i\rangle)^{\tilde{c}_i}\geq
\prod_{i=1}^k 0.49^{\tilde{c}_i}=0.49^{n+1}>n^{-2n}.
$$
Thus if $s\in [0,0.15]$ and $y\in \widetilde{\Xi}$, then
\begin{equation}
\label{XiprodRBL}
\det\left(\sum_{i=1}^k\tilde{c}_i\psi'_s(\langle y,\tilde{u}_i\rangle)\tilde{u}_i\otimes\tilde{u}_i\right)
\geq \prod_{i=1}^k\psi'_s(\langle y,\tilde{u}_i\rangle)^{\tilde{c}_i}
+\frac{\eta^4}{n^{6n+31}}.
\end{equation}
Further in  \eqref{RBLstabintermediate}, if $s\in [0,0.15]$ and $y\in \widetilde{\Xi}$, then \eqref{XiRBL1},
Lemma~\ref{gstransporterror} (ii), \eqref{cisum} and \eqref{eqeffref1} imply
\begin{equation}
\label{XiprodRBL0}
\prod_{i=1}^kg_s(\psi_s(\langle \tilde{u}_i,y\rangle))^{\tilde{c}_i} \geq
\left(\frac{1.3}{\sqrt{2\pi}}\right)^{n+1} \ge 2^{-n-1}
\geq n^{-n-1}.
\end{equation}
Applying first \eqref{BallBartheRBL}, \eqref{XiprodRBL} and \eqref{XiprodRBL0}
in  \eqref{RBLstabintermediate}, and then
\eqref{masstrans} and \eqref{XiRBL3}, we deduce that if $s\in [0,0.15]$, then
\begin{align*}
\nonumber
1+\varepsilon
&\geq \int_{\R^{n+1}}\left(\prod_{i=1}^kg_s(\psi_s(\langle \tilde{u}_i,y\rangle))^{\tilde{c}_i}\right)
\left(\prod_{i=1}^k\psi'_s(\langle \tilde{u}_i,y\rangle)^{\tilde{c}_i}\right)\,dy
+\int_{\widetilde{\Xi}}\frac{\eta^4}{n^{7n+32}}\,dy\\
\nonumber
&\geq \int_{\R^{n+1}}\left(\prod_{i=1}^kg(\langle \tilde{u}_i,y\rangle)^{\tilde{c}_i}\right)\,dy+
\frac{\eta^4}{n^{15n+31}}\\
&\geq
\left(\frac{1}{\sqrt{2\pi}}\right)^{\frac{n+1}2}\int_{\R^{n+1}}e^{- \|y\|^2/2}\,dy+
\frac{\eta^4}{n^{31n}}=1+
\frac{\eta^4}{n^{31n}}.
\end{align*}
This contradicts $\eta=n^{10n}\varepsilon^{1/4}$, and hence we conclude \eqref{uconsitionBLRBLgsg}
also in the case of the reverse Brascamp-Lieb inequality.

\medskip

Now we return to the proof of Proposition~\ref{BLRBLgsg}. Since $k\leq 2n^2$ and $\varepsilon<n^{-56n}$,
we have $\eta<1/( 6kn)$.
Since \eqref{uconsitionBLRBLgsg} is available now, we can apply  Corollary~\ref{approxort}, which
yields the existence of an orthonormal basis $\tilde{w}_1,\ldots,\tilde{w}_{n+1}$ of $\R^{n+1}$
such that $\langle e,\tilde{w}_i\rangle=\frac1{\sqrt{n+1}}$  and
$\|\tilde{u}_i-\tilde{w}_i\|\leq \angle(\tilde{u}_i,\tilde{w}_i)<6n^3\eta$ for $\eta=n^{10n}\varepsilon^{1/4}$ and $i=1,\ldots,n+1$.
Now we consider the vertices $w_1,\ldots,w_{n+1}\in S^{n-1}$ of the regular simplex which are defined by the relations
$\tilde{w}_i=\sqrt{\frac{n}{n+1}}\,w_i +\sqrt{\frac{1}{n+1}}\,e$ for $i=1,\ldots,n+1$.
Therefore
$$\angle(u_i,w_i)<\frac{\pi}{2}\|u_i-w_i\|=
\frac{\pi}{2}\,\sqrt{\frac{n+1}{n}}\,\|\tilde{u}_i-\tilde{w}_i\|<18n^3\eta<n^{14n}\varepsilon^{1/4}
$$
for $i=1,\ldots,n+1$.
In turn, we conclude Proposition~\ref{BLRBLgsg}.
\hfill \proofbox
\end{proof}

\medskip

We will actually use the Brascamp-Lieb inequality and its reverse for the function
\begin{equation}\label{gtilde}
\tilde{g}_s(t)=\mathbf{1}\{{t\geq 0}\}\exp\left( -\frac{(t-s)^2}2\right)
\end{equation}
for $s\in \R$,
where
$$
\tilde{g}_s=\left(\int_\R \tilde{g}_s\right)\, g_s.
$$
We note that if $s\geq 0$, then
\begin{equation}
\label{flambdaint}
\int_{\R}\tilde{g}_s\geq \frac{\sqrt{2\pi}}2>1.
\end{equation}
From Proposition~\ref{BLRBLgsg} and \eqref{cisum} we deduce the following strengthened version of the Brascamp-Lieb inequality and its reverse for $\tilde{g}_s$.

\begin{coro}
\label{BLRBLflambda}
Using the same notation as in Proposition~\ref{BLRBLgsg},
 let $k\leq 2n^2$, let $s\in[0,0.15]$
and let $\varrho\in(0,1)$.

If for any regular simplex with vertices $w_1,\ldots,w_{n+1}\in S^{n-1}$
and any subset $\{i_1,\ldots,i_{n+1}\}\subset\{1,\ldots,k\}$, there exists $j\in\{1,\ldots,n+1\}$ such that
$\angle(u_{i_j},w_j)\geq \varrho$, then
\begin{align*}
\int_{\R^{n+1}}\prod_{i=1}^k\tilde{g}_s(\langle x,\tilde{u}_i\rangle)^{\tilde{c}_i}\,dx&\leq
(1-n^{-56n}\varrho^4)\left(\int_\R \tilde{g}_s\right)^{n+1}, \\
{\int^*_{\R^{n+1}}}\sup_{x=\sum_{i=1}^k\tilde{c}_i\theta_i\tilde{u}_i}
\prod_{i=1}^k\tilde{g}_s(\theta_i)^{\tilde{c}_i}\,dx
&\geq  (1+n^{-56n}\varrho^4)\left(\int_\R \tilde{g}_s\right)^{n+1}.
\end{align*}
\end{coro}

\section{An almost regular simplex for Theorem~\ref{Lowner-stab}
and Theorem~\ref{meanw-isotropicstab} (a), (b)}\label{sec:7}

The entire section is devoted to proving the following statement.

\begin{prop}
\label{uiwjdist}
Let  $n+1\leq k\leq 2n^2$,  $u_1,\ldots,u_k\in S^{n-1}$ and $c_1,\ldots,c_k>0$ be such that
$$
\begin{array}{rcl}
\sum_{i=1}^kc_iu_i\otimes u_i&=&\Id_n,\\[1ex]
\sum_{i=1}^kc_iu_i&=&o,
\end{array}
$$
and $\ell(C)\geq (1-\varepsilon)\ell(\Delta_n)$ holds for $C={\rm conv}\{u_1,\ldots,u_k\}$
and $\varepsilon\in(0,n^{-60n})$.

Then for $\eta=n^{15n}\varepsilon^{\frac14}\in (0,1)$, there exists a regular simplex with vertices
$w_1,\ldots,w_{n+1}\in S^{n-1}$ and $\{i_1,\ldots,i_{n+1}\}\subset\{1,\ldots,k\}$ such that
$$
\angle(u_{i_j},w_j)\leq \eta\mbox{ \ \ for $j=1,\ldots,n+1$}.
$$
\end{prop}

We fix $e\in S^n\subset\R^{n+1}$, and
identify $e^\bot\subset \R^{n+1}$ with $\R^n$.
As   before Proposition~\ref{BLRBLgsg}, for each $u_i$, we consider
$$
\begin{array}{rcl}
\tilde{u}_i&=&\frac{\sqrt{n}}{\sqrt{n+1}}\,u_i+\frac{1}{\sqrt{n+1}}\,e\in S^n,\\[1ex]
\tilde{c}_i&=&\frac{n+1}n\,c_i,
\end{array}
$$
and hence
\begin{align}
\label{sumtildeuitensor}
\sum_{i=1}^k\tilde{c}_i\,\tilde{u}_i\otimes \tilde{u}_i&=\Id_{n+1},\\
\label{sumtildeci}
\sum_{i=1}^k\tilde{c}_i&=n+1.
\end{align}

Indirectly, we assume that Proposition~\ref{uiwjdist} does not hold, and we aim at a contradiction.
We deduce from Corollary~\ref{BLRBLflambda} and \eqref{flambdaint}
 that if $s\in[0,0.15]$, then
\begin{equation}
\label{indirect-RBLstab}
{\int^*_{\R^{n+1}}}\sup_{y=\sum_{i=1}^k\tilde{c}_i\theta_i\tilde{u}_i}
\prod_{i=1}^k\tilde{g}_s(\theta_i)^{\tilde{c}_i}\,dy
\geq  (1+n^{-56n}\eta^4)\left(\int_\R \tilde{g}_s\right)^{n+1}
>\left(\int_\R \tilde{g}_s\right)^{n+1}+n^{-56n}\eta^4.
\end{equation}

Next we provide the following general auxiliary result (which holds independently of the indirect assumption).

\begin{lemma}
\label{exponentialintestgs}
If $s\in\R$ and $C$ is defined as above, then
\begin{enumerate}
\item[{\rm (i)}]
$\displaystyle (2\pi)^{\frac{n}2}e^{-\frac{(n+1)s^2}2}\int_0^\infty e^{-\frac{r^2}2+sr\sqrt{n+1}}
\gamma_n(r\sqrt{n}C)\,dr
\geq{\int^*_{\R^{n+1}}}
\sup_{y=\sum_{i=1}^k\tilde{c}_i\theta_i\tilde{u}_i}\;
\prod_{i=1}^k\tilde{g}_s(\theta_i)^{\tilde{c}_i}\,dy, $
\item[{\rm (ii)}]
$\displaystyle
(2\pi)^{\frac{n}2}e^{-\frac{(n+1)s^2}2}\int_0^\infty e^{-\frac{r^2}2+sr\sqrt{n+1}}
\gamma_n(r\sqrt{n}\Delta_n)\,dr
=\left(\int_\R \tilde{g}_s\right)^{n+1}$.
\end{enumerate}
\end{lemma}

\begin{proof} We consider the convex cone
\begin{align*}
{\cal C}_0:&=\left\{\sum_{i=1}^k\xi_i\tilde{u}_i:\,\xi_i\geq 0\mbox{ for }i=1,\ldots,k\right\}\\
&=\left\{\sum_{i=1}^kr\sqrt{n}\lambda_iu_i+re:r\ge 0,\lambda_i\in [0,1],\sum_{i=1}^k\lambda_i=1\right\}.
\end{align*}
Then we clearly have
$x+re\in{\cal C}_0$  for $x\in\R^n$ and $r\in\R$ if and only if $r\geq 0$
and $x\in r\sqrt{n}C$.
If $y=\sum_{i=1}^k\tilde{c}_i\theta_i\tilde{u}_i$ for $\theta_1,\ldots,\theta_k\geq 0$, then
$\langle y,e\rangle=(\sum_{i=1}^k\tilde{c}_i\theta_i)/\sqrt{n+1}$, and hence
we deduce from \eqref{BLRBLquad}, \eqref{sumtildeuitensor} and \eqref{sumtildeci} that
\begin{align*}
{\int^*_{\R^{n+1}}}
\sup_{y=\sum_{i=1}^k\tilde{c}_i\theta_i\tilde{u}_i}\;
\prod_{i=1}^k\tilde{g}_s(\theta_i)^{\tilde{c}_i}\,dy&=
{\int^*_{{\cal C}_0}}
\sup_{y=\sum_{i=1}^k\tilde{c}_i\theta_i\tilde{u}_i
,  \theta_i\geq 0}\;
e^{-\frac12\sum_{i=1}^k\tilde{c}_i\theta_i^2+s\sum_{i=1}^k\tilde{c}_i\theta_i-
\frac{s^2}2\sum_{i=1}^k\tilde{c}_i}\,dy\\
&\leq e^{-\frac{(n+1)s^2}2}\int_{{\cal C}_0}
e^{-\frac12\|y\|^2+s\langle y,e\rangle\sqrt{n+1}}\,dy\\
&= e^{-\frac{(n+1)s^2}2}\int_0^\infty\int_{r\sqrt{n}C}
e^{-\frac12(\|x\|^2+r^2)+sr\sqrt{n+1}}\,dx\,dr\\
&=  (2\pi)^{\frac{n}2}e^{-\frac{(n+1)s^2}2}\int_0^\infty e^{-\frac{r^2}2+sr\sqrt{n+1}}
\gamma_n(r\sqrt{n}C)\,dr,
\end{align*}
thus we have obtained (i).

For (ii), let $w_1,\ldots,w_{n+1}$ be the vertices of $\Delta_n$, and let
$$
\tilde{w}_i=\frac{\sqrt{n}}{\sqrt{n+1}}\,w_i+\frac{1}{\sqrt{n+1}}\,e
$$
for $i=1,\ldots,n+1$.
Then $\tilde{w}_1,\ldots\tilde{w}_{n+1}$ form an orthonormal basis of
$\R^{n+1}$, and hence
$$
\sum_{i=1}^{n+1}\tilde{w}_i\otimes\tilde{w}_i=\Id_{n+1}
$$
(with $\tilde c_i=1$ in this case).
Moreover for any $y\in \R^{n+1}$, there exist unique $\theta_1,\ldots,\theta_{n+1}\in\R$
satifying $y=\sum_{i=1}^{n+1}\theta_i\tilde{w}_i$, in fact, we have $\theta_i=\langle y,\tilde{w}_i\rangle$
and  $\sum_{i=1}^{n+1}\theta_i^2=\|y\|^2$ for $i=1,\ldots,n+1$.
By the preceding argument, we deduce
\begin{align*}
(2\pi)^{\frac{n}2}e^{-\frac{(n+1)s^2}2}\int_0^\infty e^{-\frac{r^2}2+sr\sqrt{n+1}}
\gamma_n(r\sqrt{n}\Delta_n)\,dr&=
{\int^*_{\R^{n+1}}}
\sup_{y=\sum_{i=1}^k\theta_i\tilde{w}_i}\;
\prod_{i=1}^{n+1}\tilde{g}_s(\theta_i)\,dy\\
&=\left(\int_\R \tilde{g}_s\right)^{n+1},
\end{align*}
where we used Fubini's theorem for the second equality.
\hfill \proofbox
\end{proof}

\bigskip 

We apply the change of parameter $\tau=s\sqrt{n(n+1)}$
and substitution $t=r\sqrt{n}$ in Lemma~\ref{exponentialintestgs}, and conclude with the help
of the {reverse} Brascamp-Lieb inequality \eqref{RBLf} that if $\tau\in\R$, then
\begin{align*}
\nonumber
\int_0^\infty e^{-\frac{1}{2n}(t-\tau)^2}\,\gamma_n(tC)\,dt
&=e^{\frac{-s^2(n+1)}{2}}\sqrt{n}\int_0^\infty e^{-\frac{r^2}2+sr\sqrt{n+1}}
\gamma_n(r\sqrt{n}C)\,dr\\
\nonumber
&\geq\frac{\sqrt{n}}{(2\pi)^{\frac{n}2}}
{\int^*_{\R^{n+1}}}
\sup_{y=\sum_{i=1}^k\tilde{c}_i\theta_i\tilde{u}_i}
\prod_{i=1}^k\tilde{g}_s(\theta_i)^{\tilde{c}_i}\,dy\\
\nonumber
&\geq\frac{\sqrt{n}}{(2\pi)^{\frac{n}2}}
\left(\int_\R \tilde{g}_s\right)^{n+1}\\
\noindent
&=\int_0^\infty e^{-\frac{1}{2n}(t-\tau)^2}\,\gamma_n(t\Delta_n)\,dt.
\end{align*}
Hence, we get
\begin{equation}\label{intlambdatau}
\int_0^\infty e^{-\frac{1}{2n}(t-\tau)^2}\,(1-\gamma_n(t\Delta_n))\,dt
\ge
\int_0^\infty e^{-\frac{1}{2n}(t-\tau)^2}\,(1-\gamma_n(tC))\,dt.
\end{equation}
In addition, if $\tau\in[0,0.15n]\subset[0,0.15\sqrt{n(n+1)}]$, so that
$s=\tau/\sqrt{n(n+1)}\in[0,0.15)$, then
using \eqref{indirect-RBLstab}, instead of the {reverse} Brascamp-Lieb inequality \eqref{RBLf},
 we obtain
$$
\int_0^\infty e^{-\frac{1}{2n}(t-\tau)^2}\,\gamma_n(t C)\,dt \geq
  \int_0^\infty e^{-\frac{1}{2n}(t-\tau)^2}\,\gamma_n(t\Delta_n)\,dt
+\frac{\sqrt{n}}{(2\pi)^{\frac{n}2}}\, n^{-56n}\eta^4,
$$
and therefore
\begin{equation}\label{intlambdataustab}
\int_0^\infty e^{-\frac{1}{2n}(t-\tau)^2}\, (1-\gamma_n(t\Delta_n))\,dt
\ge
\int_0^\infty e^{-\frac{1}{2n}(t-\tau)^2}\,(1-\gamma_n(tC))\,dt+\frac{\sqrt{n}}{(2\pi)^{\frac{n}2}}\, n^{-56n}\eta^4.
\end{equation}
Integrating \eqref{intlambdatau} for $\tau\in\R\setminus [0,0.15n]$
and \eqref{intlambdataustab} for $\tau\in[0,0.15n]$, we deduce that
\begin{align}
\nonumber
&\int_{-\infty}^\infty\int_0^\infty e^{-\frac{1}{2n}(t-\tau)^2}\,(1-\gamma_n(t \Delta_n))\,dt\,d\tau\nonumber\\
&\qquad\qquad \geq
\int_{-\infty}^\infty\int_0^\infty e^{-\frac{1}{2n}(t-\tau)^2}\,(1-\gamma_n(t C))\,dt\,d\tau
+0.15n\, \frac{\sqrt{n}}{(2\pi)^{\frac{n}2}}\, n^{-56n}\eta^4
\label{inttaustab}
\end{align}
Since for any $t\in\R$, we have
$$
\int_{-\infty}^\infty e^{-\frac{1}{2n}(t-\tau)^2}\,d\tau=\sqrt{2\pi n}
$$
we deduce from \eqref{ellGaussianint} and \eqref{inttaustab} that
\begin{align}
\nonumber
\ell(C)&=\int_0^\infty(1-\gamma_n(tC)\,dt\\
\nonumber
&=\frac1{\sqrt{2\pi n}}\int_{-\infty}^\infty\int_0^\infty e^{-\frac{1}{2n}(t-\tau)^2}\,(1-\gamma_n(tC))\,dt\,d\tau\\
\nonumber
&\leq
\frac1{\sqrt{2\pi n}}\int_{-\infty}^\infty\int_0^\infty e^{-\frac{1}{2n}(t-\tau)^2}\,(1-\gamma_n(t
\Delta_n))\,dt\,d\tau
-\frac{0.15n}{(2\pi)^{\frac{n+1}2}}\, n^{-56n}\eta^4
\\
\label{intellstab1}
&= \ell(\Delta_n)- \frac{0.15n}{(2\pi)^{\frac{n+1}2}}\, n^{-56n}\eta^4.
\end{align}
Hence, Lemma \ref{lemrough} (a), \eqref{intellstab1}  and the hypothesis yield
$$
(1-\varepsilon)\ell(\Delta_n)\leq   \ell(C)< (1-n^{-60n}\eta^4)\ell(\Delta_n).
$$
This contradicts $\eta=n^{15n}\varepsilon^{\frac14}$, and in turn implies
Proposition~\ref{uiwjdist}.

\section{Proof of Theorem~\ref{Lowner-stab} and of Theorem~\ref{meanw-isotropicstab} (a), (b)}\label{sec:8}

For Theorem~\ref{meanw-isotropicstab}, let $\mu$ be a  centered isotropic measure  on $S^{n-1}$,
and let $K:=Z_\infty(\mu)$, and hence ${\rm supp}\,\mu=\partial K\cap S^{n-1}$.
In particular, under the assumptions of Theorem~\ref{Lowner-stab} and of Theorem~\ref{meanw-isotropicstab} (a),
we have $\ell(K)\geq (1-\varepsilon)\ell(\Delta_n )$.
First, we assume that
$$
0<\varepsilon<n^{-100n}.
$$

It follows from Lemma~\ref{discrete-iso} and John's theorem that there exist
 $k\geq n+1$ with $k\leq 2n^2$,  $u_1,\ldots,u_k\in \partial K\cap S^{n-1}$ and $c_1,\ldots,c_k>0$ such that
\begin{equation*}
\begin{array}{rcl}
\sum_{i=1}^kc_iu_i\otimes u_i&=&\Id_n,\\[1ex]
\sum_{i=1}^kc_iu_i&=&o.
\end{array}
\end{equation*}
We write $\mu_0$ to denote the centered discrete isotropic measure with
${\rm supp}\,\mu_0=\{u_1,\ldots,u_k\}$ and $\mu_0(\{u_i\})=c_i$ for $i=1,\ldots,k$, and define
$$
C:=Z_\infty(\mu_0)={\rm conv}\{u_1,\ldots,u_k\}.
$$
Since $\ell(C)\geq\ell(K)\geq (1-\varepsilon)\ell(\Delta_n)$ and $0<\varepsilon<n^{-60n}$,
it follows from Proposition~\ref{uiwjdist} that
we may assume that
the vertices $w_1,\ldots,w_{n+1}$ of $\Delta_n$ satisfy
\begin{equation}
\label{uiwieta}
\angle (u_i,w_i)\le\eta \mbox{ \ for $\eta=n^{15n}\varepsilon^{\frac14}\quad $ and $ \quad i=1,\ldots,n+1$}.
\end{equation}
For the simplex
$$
S_0={\rm conv}\{u_1,\ldots,u_{n+1}\}\subset K,
$$
we deduce from \eqref{uiwieta} and Lemma~\ref{closesimplex} (where we use $\eta<1/(2n)$)  that
$S_0^\circ\subset (1+2n\eta) \Delta_n^\circ$, and hence
\begin{equation}
\label{tildeDeltaS0}
\widetilde{\Delta}_n:=(1+2n\eta)^{-1} \Delta_n\subset S_0\subset K.
\end{equation}
We note that 
\begin{align}
0<\ell(\widetilde{\Delta}_n)-\ell(\Delta_n)&=\int_{\R^n}\|x\|_{\widetilde{\Delta}_n}\,d\gamma_n(x)-\ell(\Delta_n)\nonumber\\
&\le 
(1+2n\eta)\ell(\Delta_n)-\ell(\Delta_n)
=2n\eta\,\ell(\Delta_n).\label{tildeDeltaell}
\end{align}

\noindent{\bf Proof of Theorem~\ref{Lowner-stab}: }
Let $\xi>0$ be minimal such that 
$$
K\subset (1+\xi)\widetilde{\Delta}_n=(1+\xi)(1+2n\eta)^{-1}{\Delta}_n.
$$ 
Then
Lemma~\ref{voldifffromsimplex} and Lemma \ref{lemrough} (d) imply that
\begin{equation}
\label{VKDelta}
V(K\setminus \widetilde{\Delta}_n)\ge\frac{\xi}{n+1}\, V(\widetilde{\Delta}_n)=
\frac{\xi }{(n+1)(1+2n\eta)^n}\,V(\Delta_n)>\frac{\xi }{n^{2n+4}}\, \ell(\Delta_n).
\end{equation}

It follows from $K\subset B^n$, \eqref{tildeDeltaS0} and \eqref{VKDelta} that
\begin{align*}
\gamma_n(tK)&\ge\gamma_n(t\widetilde{\Delta}_n)\mbox{ \ for $t>0$}, \text{ and}\\
\gamma_n(tK)&\ge\gamma_n(t\widetilde{\Delta}_n)+\frac{e^{-\frac{1^2}2}}{(2\pi)^{\frac{n}2}}\, \frac{t^n\xi }{n^{2n+4}}\, \ell(\Delta_n)
\mbox{ \ for $t\in(0,1]$},
\end{align*}
and in turn we deduce from \eqref{ellGaussianint} that
$$
\ell(\widetilde{\Delta}_n)-\ell(K)\geq \int_0^1\gamma_n(tK)-\gamma_n(t\widetilde{\Delta}_n)\,dt\ge
\int_0^1\frac{t^n\xi }{n^{4n+4}}\,t \ell(\Delta_n)\,dt>n^{-(4n+5)} \ell(\Delta_n)\, \xi.
$$
We conclude from \eqref{tildeDeltaell} that
$$
(1-\varepsilon)\ell(\Delta_n)\leq \ell(K)\leq (1-n^{-(4n+5)}\xi+2n\eta)\ell(\Delta_n),
$$
and hence $\eta=n^{15n}\varepsilon^{\frac14}$ implies
$$
\xi\leq n^{4n+5}(2n\eta+\varepsilon)< n^{23n}\varepsilon^{\frac14}.
$$
It follows from \eqref{tildeDeltaS0} and the definition of $\xi$ that
\begin{equation}
\label{KDeltaSandwitch}
(1-2n\eta)\Delta_n\subset K\subset (1+\xi)\Delta_n.
\end{equation}
Since $\Delta_n\subset B^n$, $\eta=n^{15n}\varepsilon^{\frac14}< n^{23n}\varepsilon^{\frac14}=:\tilde\xi$ and
$\xi< \tilde\xi $, we conclude for the Hausdorff distance that
$\delta_H(K,\Delta_n)<n^{23n}\varepsilon^{\frac14}$.

To estimate the symmetric difference distance of $K$ and $\Delta_n$,
Lemma \ref{lemrough} (b), \eqref{KDeltaSandwitch} and $\xi< \tilde\xi\le n^{-2n}$ yield
\begin{align*}
\delta_{\rm vol}(K,\Delta_n)&\leq \left((1+\tilde\xi)^n-(1-\tilde\xi)^n\right)V(\Delta_n)\leq
2\tilde \xi n (1+\tilde\xi)^{n-1}V(\Delta_n)
< n^{25n}\varepsilon^{\frac14},
\end{align*}
which finishes the proof of  Theorem~\ref{Lowner-stab} if $\varepsilon<n^{-100n}$.
However, if $\varepsilon\geq n^{-100n}$, then Theorem~\ref{Lowner-stab} trivially holds
as $\delta_{\rm vol}(M,\Delta_n)<\kappa_n$ and $\delta_H(M,\Delta_n)<1$ for any convex body $M\subset B^n$ by the choice of
the constant $c=n^{26n}$.
\hfill \proofbox

\medskip 

\noindent{\bf Proof of Theorem~\ref{meanw-isotropicstab} (a), (b):} We assume that  $\ell(Z_\infty(\mu))\geq (1-\varepsilon)\ell(\Delta_n)$ is available.

Let $\alpha_0=9\cdot 2^{n+2}n^{2n+2}$ be the constant of Lemma~\ref{closefarsimplex}.
If for any $u\in {\rm supp}\,\mu$ there exists a $w_i$ such that $\angle(u,w_i)\leq \alpha_0\eta$, then
\begin{equation}
\label{est1}
\delta_H({\rm supp}\,\mu,\{w_1,\ldots,w_{n+1}\})\leq \alpha_0\eta<9\cdot 2^{n+2}n^{2n+2}
n^{15n}\varepsilon^{\frac14}<n^{22n}\,\varepsilon^{\frac14}.
\end{equation}
Therefore we indirectly assume that
$$
\zeta:=\max_{u\in {\rm supp}\,\mu}\min_{i=1,\ldots,n+1}\angle(u,w_i)> \alpha_0\eta,
$$
and hence there is some $u_0\in {\rm supp}\,\mu$ such that
 $\min_{i=1}^{n+1}\angle(u_0,w_i)=\zeta$. Let
$$
L:={\rm conv}\{u_0,u_1,\ldots,u_{n+1}\}.
$$
Lemma~\ref{closefarsimplex} and \eqref{uiwieta} imply
$$
V(L^\circ)\leq \left(1-\frac{\zeta}{2^{n+2}n^{2n}}  \right) V(\Delta_n^\circ).
$$
Since $L$ is a polytope with $n+2$ vertices, it is shown in
Meyer, Reisner \cite{MeS07} that
$$
V(L)V(L^\circ)\geq V(\Delta_n)V(\Delta_n^\circ),
$$
which proves a special case of the Mahler conjecture. Therefore we get
\begin{equation}
\label{VLDelta}
V(L)\geq \left(1+\frac{\zeta }{2^{n+2}n^{2n}}  \right) V(\Delta_n),
\end{equation}
while readily
$$
\widetilde{\Delta}_n\subset S_0\subset L
$$
holds for $\widetilde{\Delta}_n=(1+2n\eta)^{-1} \Delta_n$.
It follows from this, $L\subset B^n$ and \eqref{VLDelta} that
\begin{align*}
\gamma_n(tL)&>\gamma_n(t\widetilde{\Delta}_n)\mbox{ \ for $t>0$},\\
\gamma_n(tL)&>\gamma_n(t\widetilde{\Delta}_n)+\frac{e^{-\frac{1}{2}}}{\sqrt{2\pi}^{\frac{n}{2}}}\, \frac{t^n\zeta  }{2^{n+2}n^{2n}}\, V(\Delta_n)
\mbox{ \ for $t\in(0,1]$}.
\end{align*}
We deduce from \eqref{ellGaussianint} and Lemma \ref{lemrough} (d) that
$$
\ell(\widetilde{\Delta}_n)-\ell(L)\geq \int_0^1\gamma_n(tL)-\gamma_n(t\widetilde{\Delta}_n)\,dt>
\int_0^1\frac{t^n\zeta }{n^{5n+5}}\ell(\Delta_n)\,dt>n^{-8n} \ell(\Delta_n)\, \zeta.
$$
We conclude from \eqref{tildeDeltaell} that
$$
(1-\varepsilon)\ell(\Delta_n)\leq \ell(Z_\infty(\mu))\leq \ell(L)\leq (1-n^{-8n}\zeta+2n\eta)\ell(\Delta_n),
$$
and hence
$$
\zeta\leq n^{8n}(2n\eta+\varepsilon)< n^{22n}\varepsilon^{\frac14}.
$$
Therefore in both cases (compare \eqref{est1}), if $\ell(Z_\infty(\mu))\geq (1-\varepsilon)\ell(\Delta_n)$
for a centered isotropic measure $\mu$ on $S^{n-1}$, then
$$
\delta_H({\rm supp}\,\mu,\{w_1,\ldots,w_{n+1}\})<n^{28n}\,\varepsilon^{\frac14}.
$$
Since
$\ell(Z_\infty(\mu))\geq (1-\varepsilon)\ell(\Delta_n)$ and
$W(Z_\infty(\mu)^\circ)\geq (1-\varepsilon)W(\Delta_n^\circ)$ are equivalent
according to \eqref{mean-ell}, we have verified the case
$W(Z_\infty(\mu)^\circ)\geq (1-\varepsilon)W(\Delta_n^\circ)$ of Theorem~\ref{meanw-isotropicstab} 
as well, in the case $\varepsilon<n^{-100n}$.
However, if $\varepsilon\geq n^{-100n}$, then Theorem~\ref{meanw-isotropicstab} trivially holds
since for any $x\in S^{n-1}$ there exists a vertex $w$ of $\Delta_n$ with $\|x-w\|\le \sqrt{2}$.\hfill \proofbox

\section{An almost regular simplex for Theorem~\ref{John-stab}
and Theorem~\ref{meanw-isotropicstab} (c), (d)}\label{sec:9}

The whole section is dedicated to proving the following statement.

\begin{prop}
\label{uiwjdist0}
Let  $n+1\leq k\leq 2n^2$,  $u_1,\ldots,u_k\in S^{n-1}$ and $c_1,\ldots,c_k>0$ be such that
$$
\begin{array}{rcl}
\sum_{i=1}^kc_iu_i\otimes u_i&=&\Id_n,\\[1ex]
\sum_{i=1}^kc_iu_i&=&o,
\end{array}
$$
and $\ell(C^\circ)\leq (1+\varepsilon)\ell(\Delta_n^\circ)$ holds for $C={\rm conv}\{u_1,\ldots,u_k\}$
and $\varepsilon\in(0,n^{-60n})$.

Then for $\eta=n^{15n}\varepsilon^{\frac14}\in (0,1)$, there exists a regular simplex with vertices
$w_1,\ldots,w_{n+1}\in S^{n-1}$ and $\{i_1,\ldots,i_{n+1}\}\subset\{1,\ldots,k\}$ such that
$$
\angle(u_{i_j},w_j)\leq \eta\mbox{ \ \ for $j=1,\ldots,n+1$}.
$$
\end{prop}

We recall from \eqref{gtilde} that if $s\in \R$, then $\tilde g_s$ is defined by
$$
\tilde{g}_s(t)=\mathbf{1}\{{t\geq 0}\}\exp\left( -\frac{(t-s)^2}2\right),\qquad t\in\R.
$$

In this section, we slightly change the setup used in Proposition~\ref{BLRBLgsg}
and Corollary~\ref{BLRBLflambda}.
As   before Proposition~\ref{BLRBLgsg}, we fix an $e\in S^n\subset\R^{n+1}$, and
identify $e^\bot\subset \R^{n+1}$ with $\R^n$.
However now, for each $u_i\in S^{n-1}$, we consider
$$
\begin{array}{rcl}
\tilde{u}_i&=&-\frac{\sqrt{n}}{\sqrt{n+1}}\,u_i+\frac{1}{\sqrt{n+1}}\,e\in S^n,\\[1ex]
\tilde{c}_i&=&\frac{n+1}n\,c_i,
\end{array}
$$
and hence
\begin{align}
\label{sumtildeuitensor0}
\sum_{i=1}^k\tilde{c}_i\,\tilde{u}_i\otimes \tilde{u}_i&={\Id_{n+1},}\\
\label{sumtildeci0}
\sum_{i=1}^k\tilde{c}_i&={n+1,}\\
\label{sumtildeui0}
\sum_{i=1}^k\tilde{c}_i\tilde{u}_i&=\frac{\sum_{i=1}^k\tilde{c}_i}{\sqrt{n+1}}\cdot e
=\sqrt{n+1}\, e.
\end{align}

For the convex cone
$$
\widetilde{\cal C}:=\{z\in\R^{n+1}:\,\langle z,\tilde{u}_i\rangle\geq 0\mbox{ \ }i=1,\ldots,k\},
$$
the  use of $-u_i$ instead of $u_i$ in the definition of $\tilde{u}_i$ ensures that
\begin{equation}
\label{tildecalClevel}
x+r\,e\in\widetilde{\cal C}\mbox{ for  $x\in\R^n$ and $r\in \R$
 \ if and only if \ }r\geq 0\mbox{ and $x\in \frac{r}{\sqrt{n}}\,C^\circ$.}
\end{equation}
Moreover, we observe that if $C=\Delta_n$, then $k=n+1$, and $\tilde{u}_1,\ldots,\tilde{u}_{n+1}$ form an orthonormal basis of $\R^{n+1}$.

Since $-u_1,\ldots,-u_k$ satisfy the same conditions as  $u_1,\ldots,u_k$, it follows that
Corollary~\ref{BLRBLflambda} remains true for the vectors $\tilde u_1,\ldots,\tilde u_k$ as defined in this section.

\bigskip

We suppose that Proposition~\ref{uiwjdist0} does not hold, and we seek a contradiction.
From Corollary~\ref{BLRBLflambda} and \eqref{flambdaint} we deduce
 that if $s\in[0,0.15]$, then
\begin{equation}
\label{indirect-BLstab}
\int_{\R^{n+1}}\prod_{i=1}^k\tilde{g}_s(\langle z,\tilde{u}_i\rangle)^{\tilde{c}_i}\,dz\leq
(1-n^{-56n}\eta^4)\left(\int_\R \tilde{g}_s\right)^{n+1}\leq
\left(\int_\R \tilde{g}_s\right)^{n+1}-
n^{-56n}\eta^4.
\end{equation}

\bigskip

Next we state a counterpart to Lemma \ref{exponentialintestgs}, which provides general relations independent of
the indirect reasoning used to establish Proposition \ref{uiwjdist0}.

\bigskip

\begin{lemma}
\label{exponentialexpression0}
If $s\in\R$ and $C$ is defined as above, then
\begin{description}
\item{\rm (i)}
$\displaystyle (2\pi)^{\frac{n}2}e^{-\frac{(n+1)s^2}2}\int_0^\infty e^{-\frac{r^2}2+sr\sqrt{n+1}}
\gamma_n\left(\frac{r}{\sqrt{n}}\,C^\circ\right)\,dr
=\int_{\R^{n+1}}
\prod_{i=1}^k\tilde{g}_s(\langle z,\tilde{u}_i\rangle)^{\tilde{c}_i}\,dz$,
\item{\rm (ii)}
$\displaystyle
(2\pi)^{\frac{n}2}e^{-\frac{(n+1)s^2}2}\int_0^\infty e^{-\frac{r^2}2+sr\sqrt{n+1}}
\gamma_n\left(\frac{r}{\sqrt{n}}\,\Delta_n^\circ\right)\,dr
=\left(\int_\R \tilde{g}_s\right)^{n+1}$.
\end{description}
\end{lemma}

\begin{proof}
Applying first \eqref{sumtildeuitensor0}, \eqref{sumtildeci0}, \eqref{sumtildeui0} and  then  \eqref {tildecalClevel}, we obtain
\begin{align*}
\int_{\R^{n+1}}
\prod_{i=1}^k\tilde{g}_s(\langle z,\tilde{u}_i\rangle)^{\tilde{c}_i}\,dz&=
\int_{\widetilde{\mathcal{C}}}\exp\left(-\frac12\sum_{i=1}^k\tilde{c}_i\langle z,\tilde{u}_i\rangle^2
+s\sum_{i=1}^k\tilde{c}_i\langle z,\tilde{u}_i\rangle-\frac{s^2}2\sum_{i=1}^k\tilde{c}_i\right)\,dz\\
&= e^{-\frac{(n+1)s^2}2}\int_{\widetilde{\mathcal{C}}}e^{-\frac{\|z\|^2}2+s\sqrt{n+1}\langle z,e\rangle}\,dz\\
&= e^{-\frac{(n+1)s^2}2}\int_0^\infty\int_{\frac{r}{\sqrt{n}}\,C^\circ}
e^{-\frac{\|x\|^2+r^2}2+sr\sqrt{n+1}}\,dx\,dr\\
&= \displaystyle (2\pi)^{\frac{n}2}e^{-\frac{(n+1)s^2}2}\int_0^\infty e^{-\frac{r^2}2+sr\sqrt{n+1}}
\gamma_n\left(\frac{r}{\sqrt{n}}\,C^\circ\right)\,dr.
\end{align*}

For (ii), we observe that if we replace $C$ by $\Delta_n$ in the argument above, then the analogues of $\tilde{u}_1,\ldots,\tilde{u}_{n+1}$ form an orthonormal basis of $\R^{n+1}$ and $\tilde{c}_i$ is replaced by $1$.
\hfill\proofbox
\end{proof}

\medskip 

We apply the change of parameter $\tau=s\sqrt{\frac{n+1}{n}}$
and the substitution $t=r/\sqrt{n}$ in Lemma~\ref{exponentialexpression0}, and conclude with the help
of the Brascamp-Lieb inequality \eqref{BLf} that if $\tau\in\R$, then
\begin{align*}
\nonumber
\int_0^\infty e^{-\frac{n}{2}(t-\tau)^2}\,\gamma_n(tC^\circ)\,dt&=
e^{-\frac{n\tau^2}{2}}\int_0^\infty e^{-\frac{nt^2}{2}+st\sqrt{n(n+1)}}
\, \gamma_n(tC^\circ)\,dt\\
\nonumber
&=\frac{e^{\frac{-(n+1)s^2}{2}}}{\sqrt{n}}\int_0^\infty e^{-\frac{r^2}2+sr\sqrt{n+1}}
\gamma_n\left(\frac{r}{\sqrt{n}\,}C^\circ\right)\,dr\\
\nonumber
&=\frac{1}{\sqrt{n}(2\pi)^{\frac{n}2}}
\int_{\R^{n+1}}
\prod_{i=1}^k\tilde{g}_s(\langle z,\tilde{u}_i\rangle)^{\tilde{c}_i}\,dz\\
\nonumber
&\leq\frac{1}{\sqrt{n}(2\pi)^{\frac{n}2}}
\left(\int_\R \tilde{g}_s\right)^{n+1}\\
&=\int_0^\infty e^{-\frac{n}{2}(t-\tau)^2}\,\gamma_n(t\Delta_n^\circ)\,dt,\nonumber
\end{align*}
and hence
\begin{equation}\label{intlambdatau0}
\int_0^\infty e^{-\frac{n}{2}(t-\tau)^2}\,(1-\gamma_n(t\Delta_n^\circ))\,dt\le
\int_0^\infty e^{-\frac{n}{2}(t-\tau)^2}\,(1-\gamma_n(tC^\circ))\,dt.
\end{equation}
In addition, if $\tau\in[0,0.15]\subset[0,0.15\sqrt{\frac{n+1}{n}})$, which implies that $s\in[0,0.15)$,
 then
using \eqref{indirect-BLstab} instead of the Brascamp-Lieb inequality \eqref{BLf}, we obtain
\begin{align*}
\nonumber
\int_0^\infty e^{-\frac{n}{2}(t-\tau)^2}\,\gamma_n(tC^\circ)\,dt&=
\frac{1}{\sqrt{n}(2\pi)^{\frac{n}2}}
\int_{\R^{n+1}}
\prod_{i=1}^k\tilde{g}_s(\langle z,\tilde{u}_i\rangle)^{\tilde{c}_i}\,dz \\
\nonumber
&\leq
\frac{1}{\sqrt{n}(2\pi)^{\frac{n}2}}
\left( \left(\int_\R \tilde{g}_s\right)^{n+1}
- n^{-56n}\,\eta^4\right)\\
&=
\int_0^\infty e^{-\frac{n}{2}(t-\tau)^2}\,\gamma_n(t\Delta_n^\circ)\,dt
- \frac{n^{-56n}}{\sqrt{n}(2\pi)^{\frac{n}2}}\,  \eta^4.
\end{align*}
Hence, we get
\begin{equation}\label{intlambdataustab0}
\int_0^\infty e^{-\frac{n}{2}(t-\tau)^2}\,(1-\gamma_n(t\Delta_n^\circ))\,dt\le
\int_0^\infty e^{-\frac{n}{2}(t-\tau)^2}\,(1-\gamma_n(tC^\circ))\,dt
-\frac{0.15 \,n^{-56n}}{\sqrt{n}(2\pi)^{\frac{n}2}} \, \eta^4.
\end{equation}
Integrating \eqref{intlambdatau0} for $\tau\in\R\setminus [0,0.15]$
and \eqref{intlambdataustab0} for $\tau\in[0,0.15]$, we deduce that
\begin{align}
&\int_{-\infty}^\infty\int_0^\infty e^{-\frac{n}{2}(t-\tau)^2}\,(1-\gamma_n(t\Delta_n^\circ))\,dt\,d\tau
\nonumber\\
&\qquad\qquad \leq \int_{-\infty}^\infty\int_0^\infty e^{-\frac{n}{2}(t-\tau)^2}\,(1-\gamma_n(t C^\circ)\,dtd\tau
-\frac{0.15 \,n^{-56n}}{\sqrt{n}(2\pi)^{\frac{n}2}} \, \eta^4.\label{inttaustab0}
\end{align}
Since for any $t\in\R$, we have
$$
\int_{-\infty}^\infty e^{-\frac{n}{2}(t-\tau)^2}\,d\tau={\sqrt{\frac{2\pi} n},}
$$
we deduce from \eqref{ellGaussianint} and \eqref{inttaustab0} that
\begin{align*}
\nonumber
\ell(C^\circ)&=\int_0^\infty(1-\gamma_n(tC^\circ)\,dt\\
\nonumber
&=\sqrt{\frac{n}{2\pi}}\int_{-\infty}^\infty\int_0^\infty e^{-\frac{n}{2}(t-\tau)^2}\,(1-\gamma_n(tC^\circ))\,dt\,d\tau\\
\nonumber
&\geq
\sqrt{\frac{n}{2\pi}}\int_{-\infty}^\infty\int_0^\infty e^{-\frac{n}{2}(t-\tau)^2}\,(1-\gamma_n(t\Delta_n^\circ))\,dt\,d\tau
+\frac{0.15 \,n^{-56n}}{\sqrt{n}(2\pi)^{\frac{n}2}} \, \eta^4\\
&= \int_0^\infty(1-\gamma_n(t\Delta_n^\circ)\,dt
+\frac{0.15 \,n^{-56n}}{ (2\pi)^{\frac{n+1}2}} \, \eta^4
\nonumber
\\
&\ge \ell(\Delta_n^\circ)+ n^{-60n}\eta^4 \ell(\Delta_n^\circ),
\end{align*}
where  Lemma \ref{lemrough} (a) was used in the last step. This shows that 
$$
(1+\varepsilon)\ell(\Delta_n^\circ)\geq \ell(C^\circ)> (1+n^{-60n}\eta^4)\ell(\Delta_n^\circ),
$$
which contradicts $\eta=n^{15n}\varepsilon^{\frac14}$, and in turn implies
Proposition~\ref{uiwjdist0}.

\section{Proof of Theorem~\ref{John-stab} and of Theorem~\ref{meanw-isotropicstab} (c), (d)}
\label{sec:a10}

For Theorem~\ref{meanw-isotropicstab}, let $\mu$ be a  centered isotropic measure  on $S^{n-1}$,
and let $K=Z_\infty(\mu)^\circ$, and hence ${\rm supp}\,\mu=\partial K\cap S^{n-1}$.
In particular, under the assumptions of Theorem~\ref{John-stab} and of Theorem~\ref{meanw-isotropicstab} (d),
we have $\ell(K)\leq (1+\varepsilon)\ell(\Delta_n^\circ)$. First, we assume that
$$
\varepsilon<n^{-100n}.
$$

It follows from Lemma~\ref{discrete-iso} and John's theorem that there exist
 $k\geq n+1$ with $k\leq 2n^2$,  $u_1,\ldots,u_k\in \partial K\cap S^{n-1}$ and $c_1,\ldots,c_k>0$ such that
$$
\begin{array}{rcl}
\sum_{i=1}^kc_iu_i\otimes u_i&=&\Id_n,\\[1ex]
\sum_{i=1}^kc_iu_i&=&o.
\end{array}
$$
We write $\mu_0$ to denote the centered discrete isotropic measure with
${\rm supp}\,\mu_0=\{u_1,\ldots,u_k\}$ and $\mu_0(\{u_i\})=c_i$ for $i=1,\ldots,k$, and define (again)
$$
C:=Z_\infty(\mu_0)={\rm conv}\{u_1,\ldots,u_k\}\subset K^\circ.
$$
Since $\ell(C^\circ)\leq\ell(K)\leq (1+\varepsilon)\ell(\Delta_n^\circ)$,
it follows from Proposition~\ref{uiwjdist0} that
we may assume that
the vertices $w_1,\ldots,w_{n+1}$ of $\Delta_n$ satisfy
\begin{equation}
\label{uiwieta0}
\angle (u_i,w_i)\le\eta \mbox{ \ for $\eta=n^{15n}\varepsilon^{\frac14}\quad$ and $\quad i=1,\ldots,n+1$}.
\end{equation}

We observe that $K\subset S_1:=S_0^\circ$, where $S_1$ is the polar of $S_0$ and the facets of 
$$
S_1=\bigcap_{i=1}^{n+1}\{x\in\R^n:\,\langle x,u_i\rangle\leq 1\}
$$
touch $B^n$ at $u_1,\ldots,u_{n+1}$.
We deduce from \eqref{uiwieta0} and Lemma~\ref{closesimplex} that
\begin{equation}
\label{DeltanS0sandwich}
(1-n\eta) \Delta_n^\circ\subset S_1\subset (1+2n\eta) \Delta_n^\circ\subset 2\Delta^\circ_n.
\end{equation}

We claim that
\begin{equation}
\label{VKS0diff0}
\delta_{\rm vol}(K, S_1)=V(S_1\setminus K)\leq n^{23n}\varepsilon^{\frac14}.
\end{equation}
Using 
$\frac1{2n}\,S_1\subset \frac1{n}\Delta_n^\circ\subset B^n$, \eqref{ellGaussianint} and Lemma \ref{lemrough} (a), we get
\begin{align}
\nonumber
\ell(K)-\ell(S_1)&=\int_0^\infty(\gamma_n(tS_1)-\gamma_n(tK))\,dt\geq
\int_0^{\frac1{2n}} \frac{e^{-\frac{1}2}}{(2\pi)^{n/2}}\, t^nV(S_1\backslash K)\,dt\\
\label{ellKS0diff0}
&\ge
\frac{e^{-\frac{1}2}}{(n+1)(2\pi)^{n/2}(2n)^{n+1}}\, V(S_1\setminus K) \,\frac{\ell(\Delta^\circ_n)}{\sqrt{n}}>
\frac{V(S_1\setminus K)}{n^{6n}}\,\ell(\Delta^\circ_n).
\end{align}
In addition, \eqref{DeltanS0sandwich} yields
\begin{equation}
\label{ellDeltaS0diff0}
\ell(S_1)-\ell(\Delta^\circ_n)\geq\ell((1+2n\eta)\Delta^\circ_n)-\ell(\Delta^\circ_n)=
\left((1+2n\eta\right)^{-1}-1)\ell(\Delta^\circ_n)
\geq  -2n\eta \,\ell(\Delta^\circ_n).
\end{equation}
We deduce from $\ell(K)\leq (1+\varepsilon)\ell(\Delta^\circ_n)$, \eqref{ellKS0diff0} and \eqref{ellDeltaS0diff0}
that
$$
\varepsilon\,\ell(\Delta^\circ_n)\geq \ell(K)-\ell(\Delta^\circ_n)>
\left(\frac{V(S_1\setminus K)}{n^{6n}}-2n\eta\right)\,\ell(\Delta^\circ_n).
$$
Then $\eta=n^{15n}\varepsilon^{\frac14}$ implies that
$$
V(S_1\setminus K)<
n^{6n}(\varepsilon+2n\eta)<
4n\cdot n^{6n}\eta<n^{23n}\varepsilon^{\frac14},
$$
which proves \eqref{VKS0diff0}.

\bigskip

\noindent{\bf Proof of Theorem~\ref{John-stab}:} We start to deal with the symmetric volume distance of $K$ and $\Delta^\circ$.

Using \eqref{DeltanS0sandwich}, $(1+2n\eta)^n\le 4/3$ and Lemma \ref{lemrough} (b), we get
\begin{equation}
\label{voldiffS0Delta0}
\delta_{\rm vol}(S_0,\Delta^\circ_n)\leq \left((1+2n\eta)^n-(1-n\eta)^n\right)V(\Delta^\circ_n)
<n\,4n\eta\, V(\Delta^\circ_n) <n^{19n}\varepsilon^{\frac14}.
\end{equation}
Combining \eqref{VKS0diff0} and \eqref{voldiffS0Delta0}, we get
  $\delta_{\rm vol}(K,\Delta^\circ)\leq n^{24n}\varepsilon^{\frac14}$,
which proves Theorem~\ref{John-stab} (i) under the assumption $\varepsilon<n^{-100n}$.

In order to derive an upper bound for the Hausdorff distance of $K$ and $\Delta^\circ$,
we first show that the centroid $\sigma_0$ of $S_0$ satisfies
\begin{equation}
\label{S0centroid}
\sigma_0\in 4n\eta \Delta_n^\circ.
\end{equation}
To prove \eqref{S0centroid}, we observe that
$$
\Delta_n^\circ=\bigcap_{i=1}^{n+1}\{x\in\R^n:\,\langle x,w_i\rangle\leq 1\}={\rm conv}\{-nw_1,\ldots,-nw_{n+1}\}.
$$
For each $j\in\{1,\ldots,n+1\}$, \eqref{DeltanS0sandwich} yields that $S_1$ has a
vertex $v_j$ with
$$
\langle -w_j,v_j\rangle\geq h_{(1-n\eta)\Delta_n^\circ}(-w_j)=(1-n\eta)n.
$$
Since $S_1\subset (1+2n\eta) \Delta_n^\circ$ and $\Delta_n^\circ+nw_j$ is homothetic
to $\Delta_n^\circ$ with $o$ as the vertex with exterior normal $-w_j$, we have
\begin{align}
\nonumber
v_j&\in\{x\in (1+2n\eta) \Delta_n^\circ:\langle -w_j,x\rangle\geq (1-n\eta)n\}\\
\label{vjinsmallsimplex}
&\qquad\qquad =
 -(1+2n\eta)nw_j+3n\eta\left((1+2n\eta) \Delta_n^\circ+(1+2n\eta)nw_j\right)
\end{align}
for $j=1,\ldots,n+1$. Hence, the vertices $v_1,\ldots,v_{n+1}$ are contained in mutually disjoint
neighbourhoods of $-nw_1,\ldots,-nw_{n+1}$ and {thus} $S_1=\text{conv}\{v_1,\ldots,v_{n+1}\}$.

If $i=1,\ldots,n+1$, then
$\langle w_i,w_j\rangle=\frac{-1}{n}$ for $j\neq i$ implies
$\langle w_i,x\rangle\leq n+1$ for $x\in\Delta_n^\circ+nw_i$ and $\langle w_i,y\rangle\leq 0$
for $y\in \Delta_n^\circ+nw_j$ and $j\neq i$. Therefore
\eqref{vjinsmallsimplex}  yields
\begin{align*}
(n+1)\langle w_i,\sigma_0\rangle&=\langle w_i,v_i\rangle+\sum_{j\neq i}\langle w_i,v_j\rangle
\leq(1+2n\eta)[-n+3n\eta(n+1)]+n(1+2n\eta)\\
&\leq 3n(n+1)(1+2n\eta)\eta\le 4n(n+1)\eta,
\end{align*}
for $i=1,\ldots,n+1$, which proves the claim.

Note that $\sigma_0\in 4n\eta\Delta^\circ_n\subset 4n^2\eta B^n\subset \text{int}(B^n)\subset K\subset S_1$, in particular
we have $o\in\text{int}(K-\sigma_0)$ and  $K-\sigma_0\subset S_1-\sigma_0$.
Let $\xi\in[0,1)$ be minimal such that
$$
\sigma_0+(1-\xi)(S_1-\sigma_0)\subset K.
$$
From \eqref{DeltanS0sandwich}, $\eta<1/(4n^2)$ and Lemma \ref{lemrough} (c), we deduce that 
$$
V(S_1)\ge (1-n\eta)^nn^nV(\Delta_n) 
\ge \left[\left(1-\frac{1}{4n}\right)^2\left(1+\frac{1}{n}\right)\right]^{\frac{n}{2}}>1.
$$
Then it follows from Lemma~\ref{voldifffromsimplex} (i) and \eqref{VKS0diff0} that
$$
\frac{\xi^n}{e}\, V(S_1)\leq V(S_1\setminus K)<
n^{23n}\varepsilon^{\frac14}<n^{23n}\varepsilon^{\frac14}\, V(S_1),
$$
and hence $\xi<n^{24}\varepsilon^{\frac1{4n}}$. 
We deduce from \eqref{S0centroid} that $-\sigma_0\in 4n^2\eta\Delta^\circ_n$, and therefore
$$
S_1-\sigma_0\subset (1+2n\eta)\Delta^\circ_n+4n^2\eta\Delta^\circ_n\subset 2\Delta^\circ_n\subset 2nB^n.
$$
Since
$K\subset S_1\subset K+\xi(S_1-\sigma_0)$ by the definition of $\xi$, it follows that
$$
\delta_H(S_1,K)\leq 2n\xi<n^{26}\varepsilon^{\frac1{4n}}.
$$
On the other hand,  $\eta=n^{15n}\varepsilon^{\frac14}$, \eqref{DeltanS0sandwich} and
$\Delta^\circ_n\subset nB^n$ imply that
$$
\delta_H(S_1,\Delta^\circ_n)<2n^2\eta<n^{17n}\varepsilon^{\frac14}.
$$
Since $n^{17n}\varepsilon^{\frac14}<n^{26}\varepsilon^{\frac1{4n}}$
if $\varepsilon<n^{-100n}$, we have $\delta_H(K,\Delta^\circ_n)<n^{27}\varepsilon^{\frac1{4n}}$,
which completes the proof of  Theorem~\ref{John-stab} if $\varepsilon<n^{-100n}$.

However, if $\varepsilon\geq n^{-100n}$, then Theorem~\ref{John-stab} trivially holds
as $\delta_{\rm vol}(M,\Delta_n)\le n^n\kappa_n$ and $\delta_H(M,\Delta_n)\le n$ for any convex body
$M\subset nB^n$. Note that if $B^n$ is the John ellipsoid of $K$, then $K\subset nB^n$, and $\kappa_n\le 6$ for all $n\in\N$. \hfill \proofbox

\medskip 

\noindent{\bf Proof of Theorem~\ref{meanw-isotropicstab} (c), (d):} Suppose that $\ell(Z_\infty(\mu)^\circ)\leq (1+\varepsilon)\ell(\Delta_n^\circ)$.

Let $\alpha_0=9\cdot 2^{n+2}n^{2n+2}$ be the constant of Lemma~\ref{closefarsimplex}.
If for any $u\in {\rm supp}\,\mu$ there exists a $w_i$ such that $\angle(u,w_i)\leq \alpha_0\eta$, then
\begin{equation*}
\delta_H({\rm supp}\,\mu,\{w_1,\ldots,w_{n+1}\})\leq \alpha_0\eta<9\cdot 2^{n+2}n^{2n+2}
n^{15n}\varepsilon^{\frac14}<n^{22n}\,\varepsilon^{\frac14}.
\end{equation*}

Therefore we assume that
$$
\zeta:=\max_{u\in {\rm supp}\,\mu}\min_{i=1,\ldots,n+1}\angle(u,w_i)> \alpha_0\eta,
$$
and let $u_0\in {\rm supp}\,\mu$ be such that
 $\min\{\angle(u_0,w_i):i=1,\ldots,n+1\}=\zeta$. Let
$$
L={\rm conv}\{u_0,u_1,\ldots,u_{n+1}\},
$$
and hence $Z_\infty(\mu)^\circ=K\subset L^\circ\subset S_1$.
Lemma~\ref{closefarsimplex} and \eqref{uiwieta0} imply
$$
V(L^\circ)\leq \left(1-\frac{\zeta }{2^{n+2}n^{2n}}  \right)\, V(\Delta_n^\circ),
$$
thus
$$
\delta_{\rm vol}(L^\circ,\Delta_n^\circ)\geq \frac{\zeta }{2^{n+2}n^{2n}}\, V(\Delta_n^\circ).
$$
On the other hand, \eqref{DeltanS0sandwich} and $(1+2n\eta)^n<4/3$ yield
$$
\delta_{\rm vol}(S_1,\Delta_n^\circ)\leq
\left((1+2n\eta)^n-(1-n\eta)^n\right) V(\Delta_n^\circ)<4n^2\eta\, V(\Delta_n^\circ).
$$
Therefore the triangle inequality implies that
$$
V(S_1\setminus L^\circ)=\delta_{\rm vol}(S_1,L^\circ)\geq
\left(\frac{\zeta }{2^{n+2}n^{2n}}-4n^2\eta\right)V(\Delta_n^\circ).
$$
Since $V(\Delta_n^\circ)>1$ by Lemma \ref{lemrough} (c), we deduce from
\eqref{VKS0diff0}  that
$$
n^{23n}\varepsilon^{\frac14}\, V(\Delta^\circ_n)\geq
V(S_1\setminus Z_\infty(\mu)^\circ)\geq V(S_1\setminus L^\circ)>
\left(\frac{\zeta }{2^{n+2}n^{2n}}-4n^2\eta\right) V(\Delta^\circ_n).
$$
It follows from $\eta=n^{15n}\varepsilon^{\frac14}$ that
$$
\zeta<2^{n+2}n^{2n}(n^{23n}\varepsilon^{\frac14}+4n^2\eta)<n^{28n}\varepsilon^{\frac14},
$$
which proves Theorem~\ref{meanw-isotropicstab}
in the case where $\ell(Z_\infty(\mu)^\circ)\leq (1+\varepsilon)\ell(\Delta_n^\circ)$ and $\varepsilon<n^{-100n}$.

Since
$\ell(Z_\infty(\mu)^\circ)\leq (1+\varepsilon)\ell(\Delta_n^\circ)$ and
$W(Z_\infty(\mu))\leq (1+\varepsilon)W(\Delta_n)$ are equivalent
according to \eqref{mean-ell}, we have completed the proof of Theorem~\ref{meanw-isotropicstab}
if $\varepsilon<n^{-100n}$.

However, if $\varepsilon\geq n^{-100n}$, then Theorem~\ref{meanw-isotropicstab} trivially holds
as for any $x\in S^{n-1}$ there exists a vertex $w$ of $\Delta_n$ with $\|x-w\|\le\sqrt{2}$. \hfill \proofbox

\section{Proof of Corollary~\ref{mean-width-stab}}\label{asec:11}

For the proof of Corollary~\ref{mean-width-stab}, we need the following observation.

\begin{lemma}
\label{polar-Hausdorff}
If $\frac1n\,B^n\subset K,C\subset nB^n$ for  convex bodies $K$ and $C$ in $\R^n$, then
$$
\mbox{$\frac1{n^2}$}\,\delta_H(K,C)\leq \delta_H(K^\circ,C^\circ)\leq n^2\delta_H(K,C).
$$
\end{lemma}
\begin{proof} We also have $\frac1n\,B^n\subset K^\circ,C^\circ\subset nB^n$.
First, we show
\begin{equation}
\label{polar-Hausdorff0}
\delta_H(K^\circ,C^\circ)\leq n^2\delta_H(K,C).
\end{equation}
Since $K\subset C+\delta_H(K,C)B^n\subset C+n\delta_H(K,C)C=(1+n\delta_H(K,C))C$, we have
$$
C^\circ\subset (1+n\delta_H(K,C))K^\circ\subset K^\circ+n^2\delta_H(K,C)\,B^n.
$$
By symmetry, we also have $K^\circ\subset C^\circ+n^2\delta_H(K,C)\,B^n$, and thus we have
verified \eqref{polar-Hausdorff0}.

Changing the {roles} of $K,C$ and their polars $K^\circ,C^\circ$ in \eqref{polar-Hausdorff0} (and using the bipolar theorem),
we also deduce the   inequality $\delta_H(K,C)\leq n^2\delta_H(K^\circ,C^\circ)$.
\hfill \proofbox
\end{proof}

\medskip 

Since $W(K)=\mbox{$\frac2{\ell(B^n)}$}\, \ell(K^\circ) $ according to \eqref{mean-ell},
we conclude Corollary~\ref{mean-width-stab} by combining
Theorem~\ref{Lowner-stab} (ii),  Theorem~\ref{John-stab} (ii) and Lemma~\ref{polar-Hausdorff}.
\proofbox

\medskip 

\noindent {\bf Remark } The factor $n^2$ in Lemma~\ref{polar-Hausdorff} is optimal.

\bigskip

\noindent
Authors' addresses:

\medskip

\noindent
K\'aroly J. B\"or\"oczky, MTA Alfr\'ed R\'enyi Institute of Mathematics, Hungarian Academy of Sciences, Re\'altanoda u. 13-15, 1053 Budapest, Hungary. E-mail: carlos@renyi.hu

\medskip
\noindent
Ferenc Fodor, Department of Geometry, Bolyai Institute, University of Szeged, Aradi v\'ertan\'uk tere 1, 6720 Szeged, Hungary. E-mail: fodorf@math.u-szeged.hu

\medskip
\noindent
Daniel Hug, Karlsruhe Institute of Technology (KIT), D-76128 Karlsruhe, Germany. E-mail: daniel.hug@kit.edu


\begin{thebibliography}{99}

\bibitem{AGB15}
D. Alonso-Guti\'errez, J. Bastero:
Approaching the Kannan-Lovász-Simonovits and variance conjectures. Lecture Notes in Mathematics, 2131. Springer, Cham, 2015.

\bibitem{AGM15}
S. Artstein-Avidan, A. Giannopoulos, V.D. Milman: Asymptotic geometric analysis. Part I. Mathematical Surveys and Monographs, 202. American Mathematical Society, Providence, RI, 2015.

\bibitem{Bal89}
K. Ball:
Volumes of sections of cubes and related problems.
In: J. Lindenstrauss and V.D. Milman (ed), Israel seminar on Geometric
Aspects of Functional Analysis (1987--1988), 1376, Lectures Notes in
Mathematics. Springer-Verlag, 1989.

\bibitem{Bal91a}
K. Ball:
Volume ratios and a reverse isoperimetric inequality.
J. London Math. Soc. 44
(1991), 351--359.

\bibitem{Bal91b}
K. Ball:
Shadows of convex bodies.
Trans. Amer. Math. Soc. 327 (1991), 891--901.

\bibitem{Bal92}
K. Ball:
Ellipsoids of maximal volume in convex bodies.
Geom. Dedicata 41 (1992), 241--250.


\bibitem{Bal03}
K. Ball:
Convex geometry and functional analysis.
In: W.B. Johnson, L. Lindenstrauss (eds), Handbook
of the geometry of Banach spaces, 1, (2003), 161--194.

\bibitem{BaK18}
 Z. Balogh, A. Krist\'aly: Equality in Borell-Brascamp-Lieb inequalities on curved spaces.
Adv. Math. 339 (2018), 453--494.

\bibitem{Bar97}
F. Barthe: In\'egalit\'es de Brascamp-Lieb et convexit\'e.
C. R. Acad. Sci. Paris 324 (1997), 885--888.

\bibitem{Bar98}
F. Barthe:
On a reverse form of the Brascamp-Lieb inequality.
Invent. Math. 134 (1998), 335--361.

\bibitem{Bar98b}
F. Barthe:
 An extremal property of the mean width of the simplex.
Math. Ann. 310 (1998), 685--693.

\bibitem{Bar04}
F. Barthe:
A continuous version of the Brascamp-Lieb inequalities.
Geometric Aspects of Functional Analysis,
Lecture Notes in Mathematics Volume, 1850, 2004,  53--63.

\bibitem{BCE13}
F. Barthe, D. Cordero-Erausquin:
Invariances in variance estimates.
Proc. Lond. Math. Soc. 106 (2013), 33--64.

\bibitem{BCLM11}
F.~Barthe, D.~Cordero-Erausquin,  M.~Ledoux, B.~Maurey:
Correlation and Brascamp-Lieb inequalities for Markov semigroups.
Int. Math. Res. Not. 10 (2011), 2177--2216.

\bibitem{Behrend1937}
F.~Behrend: \"Uber einige Affininvarianten konvexer Bereiche. (German)
Math. Ann. 113 (1937),   713--747.

\bibitem{BBFL18}
J. Bennett, N. Bez, T.C.  Flock, S. Lee: Stability of the Brascamp-Lieb constant and applications. Amer. J. Math. 140 (2018), 543--569.

\bibitem{BCCT08}
J. Bennett, T. Carbery, M. Christ, T. Tao:
The Brascamp--Lieb Inequalities: Finiteness, Structure and Extremals.
Geom. Func. Anal. 17 (2008), 1343--1415.

\bibitem{Bor75}
C. Borell: Convex set functions in d-space. Period. Math. Hung. 6 (1975), no. 2, 111--136.

\bibitem{Bor94}
K. B\"or\"oczky, Jr.:
Some extremal properties of the regular simplex. Intuitive geometry. Colloq. Math. Soc. János Bolyai, 63, North-Holland, Amsterdam, (1994), 45--61.

\bibitem{BoH17}
K.J.~B\"or\"oczky, D.~Hug:
Isotropic measures and stronger forms of the reverse isoperimetric inequality.
 Trans. Amer. Math. Soc.  369 (2017), 6987--7019.

\bibitem{BLYZ15}
 K.J. B\"or\"oczky, E. Lutwak, D. Yang,  G. Zhang:
Affine images of isotropic measures.
J. Diff. Geom. 99 (2015), 407--442.

\bibitem{BrL76}
H.J. Brascamp, E.H. Lieb:
Best constants in Young’s inequality,
its converse, and its generalization to more than three functions.
Advances in Math. 20 (1976), 151--173.

\bibitem{CCE09}
E. Carlen, D. Cordero-Erausquin:
Subadditivity of the entropy and its relation to Brascamp-Lieb type inequalities.
Geom. Funct. Anal. 19 (2009), 373--405.

\bibitem{DLMZ02}
L. Dalla, D.G. Larman, P. Mani-Levitska, C. Zong:
The blocking numbers of convex bodies.
Discrete Comput. Geom. 24 (2002), 267--277.

\bibitem{Dis73}
V.I. Diskant:
 Stability of the solution of a Minkowski equation. (Russian)
Sibirsk. Mat. \v Z. 14 (1973), 669--673.
[Eng. transl.: Siberian Math. J. 14 (1974), 466--473.]

\bibitem{Dum10}
L. D\"umbgen:
Bounding Standard Gaussian Tail Probabilities.
arxiv:1012.2063v3

\bibitem{LFT72}
L. Fejes T\'oth:
Lagerungen in der Ebene, auf der Kugel und im Raum.
Springer-Verlag, Berlin, 2nd edition, 1972.

\bibitem{FMP09}
A. Figalli, F. Maggi, A. Pratelli:
A refined Brunn-Minkowski inequality for convex sets.
Annales de IHP (C) Non Linear Analysis 26 (2009), 2511--2519.

\bibitem{FMP10}
A. Figalli, F. Maggi, A. Pratelli:
A mass transportation approach to quantitative
isoperimetric inequalities. Inventiones Mathematicae 182, (2010), 167--211.

\bibitem{FGP}
B. Fleury , O. Gu\'edon, G. Paouris:
A stability result for mean width of $L_p$-centroid bodies. Adv. Math. 214 (2007), no. 2, 865--877.

\bibitem{FMP08}
N. Fusco, F. Maggi, A. Pratelli:
The sharp quantitative isoperimetric inequality.
 Ann. of Math. 168 (2008), no. 3, 941--980.

\bibitem{GhS17}
D. Ghilli, P. Salani:  Quantitative Borell-Brascamp-Lieb inequalities for power concave functions. J. Convex Anal. 24 (2017), 857--888.

\bibitem{GMR00}
A.A. Giannopoulos, V.D. Milman, M. Rudelson:
Convex bodies with minimal mean width. In: Geometric aspects of functional analysis,
Lecture Notes in Math., 1745, Springer, Berlin, 2000, 81--93.

\bibitem{GianPapa1999}
A.A.~Giannopoulos, M.~Papadimitrakis:
Isotropic surface area measures.
Mathematika 46 (1999), 1--13

\bibitem{Gor41}
R.D. Gordon:
Values of Mills' ratio of area to bounding ordinate and
of the normal probability integral for large values of the argument.
Annals of Mathematical Statistics 12 (1941), 364--366.



\bibitem{Groemer1990}
H.~Groemer: Stability properties of geometric inequalities. Amer. Math. Monthly 97 (1990), no. 5, 382--394.

\bibitem{Groemer1993}
H.~Groemer: Stability of geometric inequalities. Handbook of convex geometry, Vol. A, B, 125--150, North-Holland, Amsterdam, 1993.

\bibitem{GroemerSchneider1991}
 H.~Groemer, R.~Schneider: Stability estimates for some geometric inequalities. Bull. London Math. Soc. 23 (1991), no. 1, 67--74.

\bibitem{Gru07}
P.M.~Gruber: Convex and discrete geometry.
Grundlehren der
Mathematischen Wissenschaften, Springer, Berlin, 2007.


\bibitem{GrS05}
P.M. Gruber, F.E. Schuster:
An arithmetic proof of John's ellipsoid theorem.
Arch. Math. 85 (2005), 82--88.

\bibitem{Gru60}
B. Gr\"unbaum:
Partitions of mass-distributions and of convex bodies by hyperplanes,
Pacific J. Math. 10 (1960), 1257--1261.

\bibitem{GuM11}
O. Guedon, E. Milman:
Interpolating thin-shell and sharp large-deviation estimates for
  isotropic log-concave measures.
Geom. Funct. Anal. 21 (2011), 1043--1068.


\bibitem{HS11}
D. Hug, R. Schneider:
Reverse inequalities for zonoids and their application. Adv. Math. 228 (2011), 2634--2646.

\bibitem{Joh37}
F. John:
Polar correspondence with respect to a convex region.
Duke Math. J.  3  (1937),  355--369.

\bibitem{Joh48}
F. John: Extremum problems with inequalities as subsidiary conditions. In: Studies and Essays
Presented to R. Courant on His 60th Birthday, January 8, 1948, pp. 187--204,
Interscience Publishers, New York, 1948.

\bibitem{KLM95}
R. Kannan, L. Lov\'asz, M. Simonovits:
Isoperimetric problems for convex bodies and a localization  lemma.
Discrete Comput. Geom. 13 (1995), 541--559.

\bibitem{Kla09}
B. Klartag: A Berry-Esseen type inequality for convex bodies with an unconditional basis,
 Probab. Theory Related Fields {145} (2009), 1--33.

\bibitem{Kla10}
B. Klartag:
{On nearly radial marginals of high-dimensional probability measures},
J. Eur. Math. Soc., {12} (2010), 723--754.



\bibitem{LiL12}
A.-J. Li, G. Leng:
Mean width inequalities for isotropic measures.
Math. Z. 270 (2012), 1089--1110.

\bibitem{Lie90}
E.H. Lieb:
 Gaussian kernels have only Gaussian maximizers.
 Invent. Math. 102 (1990), 179--208.

\bibitem{Lit18}
A.E. Litvak:
Around the simplex mean width conjecture. In: Analytic aspects of convexity,
Springer INdAM Ser. 25, (2018), 73--84.

\bibitem{Lut93}
E. Lutwak:
 Selected affine isoperimetric inequalities.
In: Handbook of convex geometry,
North-Holland, Amsterdam, 1993, 151--176.



\bibitem{Lutwak0}
E. Lutwak, D. Yang, G. Zhang: Volume inequalities for subspaces of $L_p$.
J. Diff. Geom. 68 (2004), 159--184.

\bibitem{LYZ05}
E. Lutwak, D. Yang, G. Zhang:
$L_p$ John ellipsoids.
Proc. London Math. Soc. 90 (2005), 497--520.

\bibitem{Lutwak1}
E. Lutwak, D. Yang, G. Zhang: Volume inequalities for isotropic measures.
Amer. J. Math. 129 (2007), 1711--1723.

\bibitem{TMa17}
T. Ma:
The characteristic properties of the minimal $L_p$-mean width.
J. Funct. Spaces (2017), Art. ID 2943073, 10 pp.

\bibitem{MeS07}
M. Meyer, S. Reisner:
Shadow systems and volumes of polar convex bodies.
Mathematika 53 (2006), 129--148.

\bibitem{Mil15}
E. Milman:
On the mean-width of isotropic convex bodies and their associated $L_p$-centroid bodies.
Int. Math. Res. Not. IMRN (2015), no. 11, 3408--3423.


\bibitem{Petty1961}
C.~M.~Petty:
Surface area of a convex body under affine transformations.
Proc. Amer. Math. Soc. 12 (1961), 824--828.

\bibitem{RoS171}
A. Rossi, P. Salani: Stability for Borell-Brascamp-Lieb inequalities. Geometric aspects of functional analysis,
Lecture Notes in Math., 2169, Springer, Cham, (2017), 339--363.

\bibitem{ScS95}
G. Schechtman, M. Schmuckenschl\"ager:
A concentration inequality for harmonic measures.
In: Geometric Aspects of Functional Analysis (1992-1994), Birkhauser, (1995), 255--273.

\bibitem{Sch99}
M. Schmuckenschl\"ager:
An extremal property of the regular simplex.
In: Convex geometric analysis. Cambridge, (1999), 199--202.

\bibitem{Sch14}
R.~Schneider:
Convex bodies: the Brunn-Minkowski theory.
Cambridge,  2014.

\bibitem{Web13}
M.~Weberndorfer:
Shadow systems of asymmetric $L_p$ zonotopes.
Adv. Math. 240 (2013), 613--635.

\bibitem{Wendel48}
J.~G.~Wendel, Note on the Gamma function, Amer. Math. Monthly 55 (9) (1948), 563–-564.

\end{thebibliography}
\end{document}